\documentclass[a4paper,fleqn,leqno]{article}

\usepackage{latexsym,amssymb,amsthm,amsmath,amsfonts,mathtools,epsfig,color,epsf,graphicx}
\usepackage[english]{babel}

\usepackage[utf8]{inputenc} 
\usepackage[T1]{fontenc}  
\usepackage{lmodern}  
\usepackage{tikz}
\usetikzlibrary{trees}
\usetikzlibrary{decorations.pathreplacing}
\usepackage{tabu}
\usepackage{array,multirow,multicol}
\usepackage{hyperref}       
\usepackage{url}
\usepackage{doi}            
\usepackage{booktabs}       
\usepackage{nicefrac}       
\usepackage{microtype}      
\usepackage{lipsum}
\usepackage{enumerate}

\newcommand{\N}{\mathbb{N}}
\newcommand{\R}{\mathbb{R}}
\renewcommand{\P}{\mathbb{P}}
\newcommand{\bP}{\bar{\mathbb{P}}}
\newcommand{\Z}{\mathbb{Z}}
\newcommand{\E}{\mathbb{E}}
\newcommand{\Q}{\mathbb{Q}}
\newcommand{\X}{\mathcal{X}}
\newcommand{\fe}{\mathfrak{e}}
\newcommand{\cE}{\mathcal{E}}
\newcommand{\cA}{\mathcal{A}}
\newcommand{\fR}{\mathfrak{R}}

\newcommand{\fK}{\mathfrak{K}}

\newcommand{\sX}{\mathcal{X}^\fe}
\newcommand{\sZ}{Z^\fe}
\renewcommand{\d}{{\rm d}}
\newcommand{\e}{{\rm e}}

\newcommand{\corrected}{\color{black}} 
\newcommand{\footremember}[2]{%
	\footnote{#2}
	\newcounter{#1}
	\setcounter{#1}{\value{footnote}}%
}

\makeatletter
\@addtoreset{equation}{section}
\makeatother

\newcounter{extralabel}[section]
\newcounter{assumption}

\newtheorem{ittheorem}{Theorem}
\newtheorem{itlemma}{Lemma}
\newtheorem{itproposition}{Proposition}
\newtheorem{itdefinition}{Definition}
\newtheorem{itcorollary}{Corollary}
\newtheorem{itconjecture}{Conjecture}
\theoremstyle{definition}
\newtheorem{itassumption}{Assumption}
\newtheorem{itremark}{Remark}
\newtheorem{itexample}{Example}

\newenvironment{theorem}{\addtocounter{extralabel}{1}
	\begin{ittheorem}}{\end{ittheorem}}

\newenvironment{corollary}{\addtocounter{extralabel}{1}
	\begin{itcorollary}}{\end{itcorollary}}

\newenvironment{lemma}{\addtocounter{extralabel}{1}
	\begin{itlemma}}{\end{itlemma}}

\newenvironment{proposition}{\addtocounter{extralabel}{1}
	\begin{itproposition}}{\end{itproposition}}

\newenvironment{definition}{\addtocounter{extralabel}{1}
	\begin{itdefinition}}{\end{itdefinition}}

\newenvironment{remark}{\addtocounter{extralabel}{1}
	\begin{itremark}}{\end{itremark}}

\newenvironment{assumption}{\addtocounter{assumption}{1}
	\begin{itassumption}}{\end{itassumption}}
\newenvironment{example}{\addtocounter{extralabel}{1}
	\begin{itexample}}{\end{itexample}}

\title{Spatial populations with seed-banks in random environment:\\
{III.} Convergence towards mono-type equilibrium}
\author{Shubhamoy Nandan\footremember{alley}{Mathematisch Instituut, Universiteit Leiden, Niels Bohrweg 1, 2333 CA  Leiden, NL.\newline \hspace*{0.56cm}E-mail: s.nandan@math.leidenuniv.nl}}
\date{}
\begin{document}
	\maketitle
\begin{abstract}
We consider the spatially inhomogeneous Moran model with seed-banks introduced in \cite{HN01}. Populations comprising \emph{active} and \emph{dormant} individuals are spatially structured in colonies labelled by $\mathbb{Z}^d,\,d\geq 1$. The population sizes are sampled from an \emph{ergodic}, \emph{translation-invariant}, \emph{uniformly elliptic} field that constitutes a \emph{random environment}. Individuals carry one of two types: $\heartsuit$ and $\spadesuit$. Dormant individual resides in what is called a \emph{seed-bank}. Active individuals \emph{exchange} type from the seed-bank of their own colony, and \emph{resample} type by choosing a parent uniformly at random from the distinct active populations according to a symmetric migration kernel. In \cite{HN01} by exploiting a \emph{dual} process given by an \emph{interacting coalescing particle system}, we showed that the spatial system exhibits a dichotomy between \emph{clustering} (mono-type equilibrium) and \emph{coexistence} (multi-type equilibrium). In this paper, we identify the domain of attraction for each mono-type equilibrium in the clustering regime for an \emph{arbitrary fixed} environment. Furthermore, we show that in dimensions $d\leq 2$, when the migration kernel is \emph{recurrent}, for almost all realization of the environment, the system with an \emph{initially consistent} type-distribution converges weakly to a mono-type equilibrium in which the probability of fixation to the all type-$\heartsuit$ configuration  does not depend on the environment. An explicit formula for the fixation probability is given in terms of an annealed average of the type-$\heartsuit$ densities in the active and the dormant population, biased by the ratio of the two population sizes at the target colony. 

Primary techniques employed in the proofs include stochastic duality and the environment process viewed from particle, introduced in \cite{Goldshied19} for random walk in random environment on a strip. A spectral analysis of Markov operator yields quenched weak convergence of the environment process associated with the \emph{single-particle dual process} to a reversible ergodic distribution, which we transfer to the spatial system of populations by using duality.

		\medskip\noindent
		\emph{Keywords:} 
		Moran model, resampling, migration, seed-bank, random environment, fixation probability, clustering, coexistence, duality, interacting particle system.
		
		\medskip\noindent
		\emph{MSC 2020:} 
		Primary 
		60K35; 
		Secondary 
		92D25. 
		
		\medskip\noindent 
		\emph{Acknowledgements:} 
		The research in this paper was supported by the Netherlands Organization for Scientific Research (NWO) through grant TOP1.17.019. The author thanks his advisor Frank den Hollander for suggesting the problem and appreciates his continuous encouragement during many inspiring discussions. The author thanks Evgeny Verbitskiy for several discussions on ergodic theory in the past, and for pointing out the reference \cite{Cohen14} which was very helpful for the present paper. The author also thanks Yuval Peres for providing (at online forum math.stackexchange.com) the reference \cite{Lin82} that relates weak convergence of Markov chains to the peripheral spectrum of Markov operator. The author thanks Rajat Hazra for insightful discussions.  Furthermore, the author thanks Frank Redig, Cristian Giardin\`a and Simone Floreani for discussions on duality.
		
\end{abstract}

\newpage
\tableofcontents
\newpage
\section{Introduction}
\paragraph{Background and literature.}
In recent years, understanding evolutionary behaviour of microbial populations that maintain a \emph{seed-bank}, or other \emph{dormant forms}, has gained considerable attention from both biologists and mathematicians \cite{BCEK15,BCKW16,GHO1,GHO3}. Dormancy refers to the ability of an organism to enter into a reversible state of reduced metabolic activity in response to adverse environmental conditions. While dormant, organisms refrain from reproduction, and other phenotypic development, until they become active again. While dormancy is a trait found mostly in microbial populations, the natural analogue of dormancy in plant populations is the suspension of seed germination in difficult ecological circumstances. Several experiments suggest that populations exhibiting dormancy have better heterogeneity, survival fitness and resilience \cite{Tellier11,Ivkovi12}. Dormancy appears to be ubiquitous to many forms of life and is considered to be an important evolutionary trait \cite{LJ11,LS}. Although the direct effect of this trait is not easily detected when viewed on the evolutionary time scale, researchers have made various attempts to better understand it from a mathematical perspective (see e.g. \cite{LHBB,BK} for a broad overview).

In a stochastic individual-based model, dormancy is mathematically incorporated by turning off reproduction or \emph{resampling} for a random and possibly extended period of time. This way of modelling dormancy introduces memory, and thereby gives rise to a rich behaviour of the underlying stochastic system. The first mathematical model dealing with the effect of dormancy goes back to \cite{Cohen66}. Since then several other ways to model seed-banks mathematically have emerged \cite{Kaj2001,BCEK15,BCKW16}. For example, in the model proposed in \cite{Kaj2001}, the classical Fisher--Wright model was extended to include a \emph{weak} seed-bank, where individuals reproduce offspring several generations ahead in time, with the skipped generations being interpreted as a dormant period for the offspring. However, in this model the resulting genealogy of the population, albeit stretched over time, retains the same coalescent structure described by the so-called Kingman coalescent process. Different qualitative behaviour was observed in \cite{Blath13,Blath19} by including a \emph{strong} seed-bank component, which enables the dormant individuals to have wake-up times with fat tails.
A trade-off in these models was the loss of the Markov property in the time-evolution of the system. This issue was partially tackled in \cite{BCKW16}, which introduced the \emph{seed-bank coalescent}, a new class of coalescent structures that, broadly speaking, describes the genealogy of a population exhibiting extreme dormancy.

All previously mentioned models study the effect of dormancy in a single-colony population and are mainly concerned with the underlying genealogy. Seed-bank models dealing with geographically structured populations are rare, and mathematically rigorous results are still under development. Only recently, in \cite{GHO1} (see also \cite{HP}), existing seed-bank models were extended to the \emph{spatial} setting by incorporating \emph{migration} of individuals between different colonies. These works overcome the challenge of modelling seed-banks with fat-tailed exit times by adding internal layers to the seed-banks, where \emph{active} individuals acquire a \emph{colour} before entering into a layer of the seed-bank that determines the wake-up time. Three different seed-bank models of increasing generality were introduced. A full description of the different regimes in the long-time behaviour of these models was obtained in \cite{GHO1} for the geographic space $\Z^d,\,d\geq 2$, whereas a multi-scale \emph{renormalization} analysis on the hierarchical group was carried out in \cite{GHO3}. Moreover, the finite-systems scheme was established \cite{Oomen22} as well (i.e., how a truncated
version of the system behaves on a properly tuned time scale as the truncation level tends to
infinity).

Whilst the works cited so far have dealt with seed-bank models only in the \emph{diffusive regime} which is obtained after taking the \emph{large-colony-size-limit} of individual-based models, it is natural and biologically more reasonable to consider seed-bank models with populations that have \emph{finite} sizes. However, a key challenge in dealing with stochastic models of finite populations evolving under evolutionary forces such as dormancy, resampling, migration etc., is the absence of mathematical tools for carrying out sophisticated computations. Recently, stochastic duality \cite{CGGR,JK2014} has proven to be a formidable tool for performing exact computations in many stochastic interacting systems. In particular, \cite{HN01} analyses using duality a stochastic individual-based model that incorporates dormancy in a spatial system of \emph{finite} populations. To the best of our knowledge, the combined effect of evolutionary forces such as dormancy, resampling and migration in the finite population setting has not been studied in the literature before.
\paragraph{Motivations and targets.}
The model introduced in \cite{HN01} consists of geographically structured population with preassigned finite sizes and is described via an \emph{interacting particle system} evolving in an \emph{inhomogeneous} state space. Individuals live in colonies labelled by $\Z^d,\, d\geq 1$, carry one of two genetics types: $\heartsuit,\ \spadesuit$, and can be either active or dormant. While active, individuals resample not only from their own colony, but also from other colonies according to a \emph{symmetric random walk transition kernel}. The latter is referred to as migration. Active individuals exchange types with dormant individuals of their own colony. The sizes of the active and the dormant populations depend on the colony and remain constant throughout the time evolution of the system. The underlying genealogy of the spatial populations is described by an \emph{interacting structured seed-bank coalescent}, referred to as the \emph{dual}, where lineages switch between an active and a dormant state, and perform interacting coalescing random walks on the geographic space $\Z^d$. By exploiting the dual, it was shown that the spatial system exhibits a dichotomy between \emph{clustering} (= existence of mono-type equilibrium) and \emph{coexistence} (= existence of multi-type equilibrium). Further in \cite{HN01} convergence of the spatial process to an equilibrium was established only for a restricted class of initial distributions, and in \cite{HN02} refined conditions on the model parameters were derived for which the system exhibits clustering, i.e., any attained equilibrium is mono-type. In particular, it was proved that if the relative strengths of the seed-banks, i.e., the ratios of the dormant and the active populations are bounded uniformly over the geographic space, then clustering is equivalent to the symmetric random walk kernel being recurrent. Some further lines of prospective research are:
\begin{itemize}
	\item[(1)] Identify the \emph{domain of attraction} of each equilibrium in the clustering and the coexistence regime.
	\item[(2)] Identify the parameter regime for coexistence when the relative strengths of the seed-banks are unbounded or have \emph{infinite mean}.
	\item[(3)] Establish finite-system schemes in the coexistence regime and quantify the rate of cluster growth in the clustering regime.
\end{itemize}
In the present paper we study the spatial model with seed-banks by treating the preassigned constant population sizes as an \emph{environment} of the system.
One of our main contributions is that we provide a full characterization of the domain of attraction for each mono-type equilibrium in the clustering regime for an \emph{arbitrary fixed} environment (satisfying mild regularity conditions).

In the spatial model described above the constituent populations maintain constant sizes over time. While this can be biologically explained by assuming that the system receives sufficient supply of environmental resources, a more natural extension would be to consider the model where population sizes come from a \emph{random field} determined by environmental factors such as extreme temperatures, inadequate supply of food resources, etc. Research in this direction has started only recently (see e.g. \cite{Maite19,Blath21,Wisnoski22}), although most results are available only for models that are scaled diffusively or are simulation based.

The novelty of the present paper is that we study the mono-type equilibrium behaviour of the spatial system with seed-banks introduced in \cite{HN01} for the setting where the population sizes constitute a \emph{random static} environment. In particular, the sizes are drawn from an \emph{ergodic} and \emph{translation-invariant} random field. Our contributions are two-fold:
\begin{itemize}
	\item[(a)] When the symmetric migration kernel is \emph{recurrent} (which requires $d\leq 2$) and the random environment is \emph{uniformly elliptic}, we show that the system started from an initially \emph{consistent} type distribution converges in law to a mono-type equilibrium for almost all realization of the environment. In other words, we prove that the system undergoes \emph{homogenization} in the \emph{quenched} setting.
	\item[(b)] We show that, in the homogenized mono-type equilibrium, the \emph{fixation probability} (in law) to the all type-$\heartsuit$ configuration is deterministic, i.e., does not depend on the realization of the environment. We also provide an explicit formula for this probability.
\end{itemize}
The techniques used in the proof of the main theorems include stochastic duality, moment relations, semigroup expansion and the environment viewed from the particle recently introduced in \cite{Goldshied19} for random walk in random environment on a strip, and spectral analysis of Markov kernel operator.
\paragraph{Outline.} The paper is organized as follows. In Section~\ref{sec-main-theorems} we give a definition of the spatial model, state our main theorems on the convergence of the system to a mono-type equilibrium, and explain the strategy of the proofs in detail. Section~\ref{sec-single-particle-proc-in-random-env} is devoted to the analysis of dual process with a single lineage (or single particle) in random environment, where homogenization results are derived for the associated \emph{environment process}. In Section~\ref{sec-proof-of-main-theorem} we prove our main theorems using the results derived in Section~\ref{sec-single-particle-proc-in-random-env}. In Appendix~\ref{sec-stationarity-n-large-numbers}, we prove a result stated in Section~\ref{sec-single-particle-proc-in-random-env} on the existence of a stationary distribution for the aforementioned environment process, and also give a proof of the \emph{strong law of large numbers} for the single-particle dual, which is a result of independent interest. Finally, in Appendix~\ref{sec-fundamental-theorem-of-MC} we prove an auxiliary proposition relating weak convergence of Markov chain to the peripheral point-spectrum of a Markov operator, which is needed for the proof of our main theorems.
\section{Main theorems}
\label{sec-main-theorems}
In Section~\ref{sec-defn-of-system} we recall the spatially inhomogeneous system of populations with seed-banks from the companion paper \cite{HN01} and set the stage to state our main results. For a formal mathematical description of the spatial system, we refer the reader to \cite[Section 3.2]{HN01}. In Section~\ref{sec-clustering-in-fixed-env} we give our first main result on the convergence of the system in the clustering regime for an \emph{arbitrary fixed environment} (Theorem~\ref{thm:domain-of-attr-cluster}). In Section~\ref{sec-clustering-in-random-env} we consider the system in a \emph{static random} environment that is drawn from an ergodic and translation-invariant field defined on a subset of \emph{uniformly elliptic} environments, and present a \emph{homogenization} statement in the \emph{quenched} setting on the convergence of the system to a mono-type equilibrium (Theorem~\ref{thm:random-env-clustering}--\ref{thm:weak-convergence-of-contin-env}). In Section~\ref{sec-discussion} we discuss the results and shed light on the strategy of the proofs.
\subsection[Quick definition of the model]{Quick recount of the model and basic notations}
\label{sec-defn-of-system}
We consider the integer lattice $\Z^d, d\geq 1,$ as a \emph{geographic space}, where each $i\in\Z^d$ represents a colony consisting of two kinds of population: \emph{active} and \emph{dormant}. For $i\in\Z^d,$ we write $(N_i,M_i)\in\N^2$ to denote the \emph{size} of the active, respectively, the dormant population at colony $i$. The sizes of the populations are preassigned and can vary across different colonies. Further, each individual inside a population carries one of two \emph{genetic types}: $\heartsuit$ and $\spadesuit$. Individuals in the active (resp.\ dormant) populations of the spatial system (see Fig.~\ref{fig:representation-of-spatial-system}) are called active  (resp.\ dormant). Individuals update their genetic type over time:
\begin{itemize}
	\item[(1)]
	Active individuals in any colony \emph{resample}  active individuals in any colony. 
	\item[(2)]
	Active individuals in any colony \emph{exchange} with dormant individuals in the same colony.
\end{itemize} 
For (1) we assume that each active individual at colony $i$ at rate $a(i,j)$ uniformly draws an active individual at colony $j$ and \emph{adopts its type}. For (2) we assume that each active individual at colony $i$ at rate $\lambda \in (0,\infty)$ uniformly draws a dormant individual at colony $i$ and the two individuals \emph{trade places while keeping their type} (i.e., the active individual becomes dormant and the dormant individual becomes active). Dormant individuals  do \emph{not} resample and thereby cause an overall slow-down of the random genetic drift that arises from (1). Because of this, we refer to the dormant populations as the \emph{seed-banks} of the spatial system. Although the exchange rate $\lambda$ could be made to vary across colonies, for the sake of simplicity we choose it to be constant.

\medskip\noindent
We put
\begin{equation}
	\label{eqn:defn-of-Ki}
	K_i:=\frac{N_i}{M_i}, \qquad i\in\Z^d, 
\end{equation}
for the \emph{ratios} of the sizes of the active and the dormant population in each colony. Observe that $K_i^{-1}$ quantifies the \emph{relative strength} of the seed-bank at colony $i\in\Z^d$. 

\medskip\noindent
We impose the following conditions on the \emph{migration kernel} $a(\cdot\,,\cdot)$:
\begin{assumption}{\bf [Homogeneous migration]}
	\label{assumpt1}
	The migration kernel $a(\cdot\,,\cdot)$ satisfies:
	\begin{itemize}
		\item[\rm{(1)}]
		$a(\cdot\,,\cdot)$ is irreducible in $\Z^d$.
		\item[\rm{(2)}]
		$a(i,j) = a(0,j-i)$ for all $i,j\in{\Z^d}$.
		\item[\rm{(3)}]
		$c:=\displaystyle\sum_{i\in\Z^d\backslash\{0\}} a(0,i) < \infty$ and $a(0,0) > 0$.
	\end{itemize}
\end{assumption}

\noindent
Part (2) ensures that the way genetic information moves between colonies is homogeneous in space. Part (3) ensures that the total rate of resampling of a single individual is finite and that resampling is possible also at the same colony. 
\begin{figure}[h!t!]
	\vspace{0.5cm}
	\begin{center}
		\includegraphics[width=12cm,height=4.5cm]{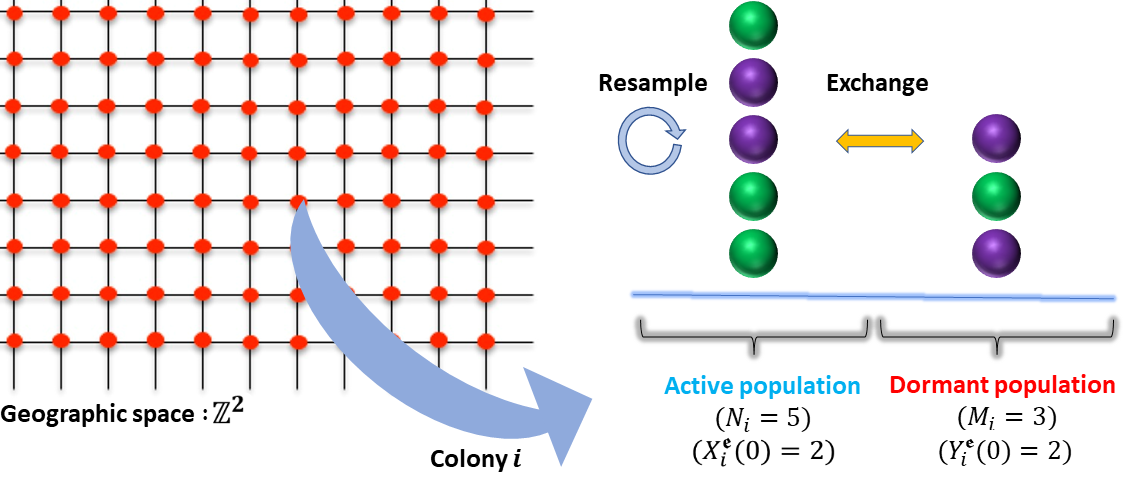}
	\end{center}
	\caption{\small A schematic representation of the spatial populations on geographic space $\Z^2$ in an environment $\fe:=(N_k,M_k)_{k\in\Z^2}$. Purple individuals are of type $\heartsuit$ and green individuals are of type $\spadesuit$. The active (resp.\ dormant) population at colony $i$ has size $N_i=5$ (resp.\ $M_i=3$). The system evolves in time under the influence of resampling and exchange. }
	\label{fig:representation-of-spatial-system}
\end{figure}

\medskip\noindent
Under the resampling and exchange dynamics described earlier, the initial population sizes $(N_i,M_i)_{i\in\Z^d}$ remain constant over time. Thus, we can naturally think of the sizes of the populations as a \emph{static environment} for the spatial system. Throughout the sequel we denote by $\fe:=(N_i,M_i)_{i\in\Z^d}\in(\N\times\N)^{\Z^d}$ a typical choice for the sizes of the constituent populations and refer to it as the \emph{environment}.
For $n\in\N$, we write $[n] := \{0,1,\ldots,n\}$, and at each colony $i$ we register the pair $(X_i^\fe(t),Y_i^\fe(t))\in[N_i]\times[M_i]$, representing the number of active, respectively, dormant individuals of type $\heartsuit$ at time $t$ at colony $i$.

The resulting Markov process is an interacting particle system denoted by
\begin{equation}
	\label{process}
	\sZ:=(Z^\fe(t))_{t \geq 0}, \qquad \sZ(t) := ( X_i^\fe(t),Y_i^\fe(t))_{i \in \Z^d},
\end{equation}
and lives on the state space
\begin{equation}
	\label{eqn:defn-of-Z-state-space}
\sX := \prod_{i\in\Z^d} [N_i]\times[M_i].
\end{equation}
It is implicitly assumed that the state space $\sX$ is equipped with the natural product topology, under which it becomes compact. The space of càdlàg functions on $\sX$ is endowed with the Skorokhod topology and plays the role of the ambient probability space for the process $\sZ$. The superscript $\fe$ indicates the dependence of the process $\sZ$ on the environment $\fe=(N_i,M_i)_{i\in\Z^d}$.

Throughout the sequel we adopt the convention of adding a superscript (or subscript) with Fraktur font to emphasize the dependence of a variable on the realization of the environment. Furthermore, in order to avoid unnecessary technicalities, we throughout consider environments that are admissible in the following sense:

\begin{definition}{\bf [Admissible environments]}
	\label{defn:admissible-colony-sizes}
	 \rm{Consider the following three conditions for the environment $\fe = (N_i,M_i)_{i\in\Z^d}\in (\N\times\N)^{\Z^d}$ and the migration kernel $a(\cdot\,,\cdot)$:
	\begin{itemize}
		\item[\rm{(a)}]\label{assump:non-trivial colony-sizes}
		$N_i\geq 2$ and $M_i\geq 2$ for all $i\in\Z^d$.
		\item[\rm{(b)}]
		$\sup_{i\in\Z^d\backslash\{0\}} \|i\|^{-\gamma} N_i < \infty$ and $\sum_{i\in\Z^d} \|i\|^{d+\gamma+\delta} a(0,i)<\infty$ for {some} $\gamma>0$ and some $\delta>0$.
		\item[\rm{(c)}]
		$\lim_{\|i\| \to \infty} \|i\|^{-1} \log N_i = 0$ and $\sum_{i\in\Z^d} \e^{\delta \|i\|} a(0,i)<\infty$ for some $\delta>0$.
	\end{itemize}
	If (a) is satisfied, i.e., in each colony, both the active and the dormant population consist of at least two individuals, then we say that $\fe$ is \emph{non-trivial}. Further, if either (b) or (c) is satisfied, then we say that $\fe$ is \emph{compatible}. Non-trivial and compatible environments are referred to as \emph{admissible} environments. The set of all admissible environments is denoted by $\cA$.
$\hfill\blacklozenge$
}
\end{definition}
In \cite[Theorem 2.2]{HN01} it was shown {\corrected by formulating a well-posed martingale problem} that under Assumption~\ref{assumpt1}, for any compatible environment $\fe$, the Markov process $\sZ$ in \eqref{process} is well-defined.
\begin{remark}{\bf[Higher moments]}
\label{rem:higher-moments-of-migration-kernel}
{\rm Unfortunately, because of conditions (b) and (c) in Definition~\ref{defn:admissible-colony-sizes}, the migration kernel $a(\cdot\,,\cdot)$  is required to have at least $d+\delta$ finite moment for some $\delta>0$. We believe that this can be relaxed to a weaker moment condition. {\corrected On the one hand, these moment conditions are naturally required to ensure the well-posedness of the martingale problem associated with the process $Z^\fe$. On the other hand, if there is a \emph{uniform upper bound} available for the active population sizes $(N_i)_{i\in\Z^d}$, then Assumption~\ref{assumpt1} alone is sufficient to carry out the construction of $Z^\fe$ by following Liggett's method based upon the Hille-Yosida theory of semigroups. As a matter of fact, in the latter case, the conditions stated in \cite[Chapter 1, Theorem~3.9]{L} are met, and therefore, it is possible to avoid the method of well-posed martingale problem adopted in \cite[Section 3.2.3]{HN01} altogether, and to drop conditions (b)--(c) from Definition~\ref{defn:admissible-colony-sizes}. However, if the active population sizes are unbounded, then Liggett's method does not work straight away, but the method of well-posed martingale problem succeeds under either of the above two conditions.} We do not require a growth restriction on the sizes of the dormant populations, because in our model only active individuals initiate resampling and exchange of the types, while dormant individuals sit idle. Condition (a) arises from a technical requirement in \cite{HN02} and may be removed with minor adaptations.}
\end{remark}
\subsection{Clustering in a fixed environment}
\label{sec-clustering-in-fixed-env}
A natural question that arises in the discussion of any model is whether an equilibrium exists. To answer this, let us denote by $\mathcal{P(\sX)}$ the set of all probability distributions on $\sX$, and let $\delta_\fe\in\mathcal{P}(\sX)$ (resp., $\delta_\mathbf{0}$) be the Dirac distribution concentrated at the configuration $\fe\in\sX$ (resp., $(0,0)_{i\in\Z^d}$). Observe that $\fe\in\sX
$ (resp., $(0,0)_{i\in\\Z^d}$) is the configuration where all individuals are of type-$\heartsuit$ (resp., type-$\spadesuit$), and therefore $\delta_\fe,\delta_\mathbf{0}$ are two trivial extremal equilibria for the process $\sZ$. Indeed, when all individuals in the spatial system have the same genetic type, neither resampling nor exchange can reintroduce the missing genetic type, and thereby push the system to an out-of-equilibrium state. This immediately raises the question of existence of any other equilibrium apart from these two trivial ones, and is the reason for introducing the following definition:

\begin{definition}{\bf [Clustering and Coexistence]}
	\label{defn:clustering-coexistence}
{\rm  We say that the process $\sZ$ is in the \emph{clustering} regime if $\delta_\mathbf{0}$ and $\delta_\fe$ are the only two extremal equilibrium. Otherwise, we say that the process is in the \emph{coexistence} regime.
}\hfill$\blacklozenge$
\end{definition}
\begin{remark}
	{In the clustering regime any equilibrium $\nu\in\mathcal{P}(\sX)$ of the process $\sZ$ is a mixture of $\delta_\mathbf{0}$ and $\delta_\fe$. Thus, if the process $\sZ$ exhibits clustering and is in equilibrium, all individuals in the spatial system are of type $\heartsuit$ or of type $\spadesuit$.}
\end{remark}
\noindent
In \cite[Theorem~3.17]{HN01}, a necessary and sufficient criterion for clustering was formulated in terms of a \emph{dual} $(Z^\fe_*(t))_{t \geq 0}$ of the process $\sZ$. The dual process
\begin{equation}
	\label{defn:dual-process}
	Z^\fe_*:=(Z^\fe_*(t))_{t\geq 0},\quad Z^\fe_*(t):= (n_i^\fe(t),m_i^\fe(t))_{i\in\Z^d},
\end{equation}
is also an interacting particle system, {\corrected which lives on the state space
\begin{equation}
	\label{eqn:defn-of-dual-state-space}
	\X^{\fe}_* := \Big\{(n_i,m_i)_{i\in\Z^d} \in \sX\colon\,\sum_{i\in\Z^d}(n_i+m_i)<\infty\Big\}
\end{equation}
consisting of configurations in $\sX$ with finite masses. It describes a Markovian evolution of a finite collections of indistinguishable particles that switch between an \emph{active} and a \emph{dormant} state.} The variable $n_i^\fe(t)$ (resp.\ $m_i^\fe(t)$) in \eqref{defn:dual-process} counts the number of active (resp.\ dormant) dual particles present at location $i\in\Z^d$ at time $t\geq 0$. The dual particles perform \emph{interacting coalescing} random walks on $\Z^d$ as long as they are in the active state, with rates (see \cite[Definition 3.7]{HN01}) that are determined by the environment $\fe$, the migration kernel $a(\cdot\,,\cdot)$ and the exchange rate $\lambda$.

{\corrected In \cite{HN02} the clustering criterion given in \cite[Theorem~3.17]{HN01} was further refined by exploiting a two-particle version of the dual}, and conditions on the environment $\fe$ and other parameters were obtained for which the process $\sZ$ exhibits clustering. In particular, it was shown (see \cite[Corollary~2.14]{HN02}{\corrected \footnote{\corrected Corollary~2.14 follows from \cite[Theorem~2.13]{HN02} that contained a minor gap in its proof. The issue has been resolved by Frank den Hollander and the present author with the help of a zero-one law in \cite[Lemma~A.2.2]{Nandan2023}.}}) that clustering prevails under the following set of conditions:
\begin{assumption}{\bf [Clustering environment]}
	\label{assumpt:clustering-env}
	The migration kernel $a(\cdot\,,\cdot)$ satisfying Assumption~\ref{assumpt1} and the environment $\fe=(N_i,M_i)_{i\in\Z^d}$ {\corrected (admissible in the sense of Definition~\ref{defn:admissible-colony-sizes})} are such that
	\begin{itemize}
		\item[\rm{(1)}] $a(\cdot\,,\cdot)$ is symmetric, i.e.,
		\begin{equation}
			a(0,i) = a(0,-i),\quad i\in\Z^d.
		\end{equation}
		\item[\rm{(2)}] $a(\cdot\,,\cdot)$ generates a recurrent random walk on $\Z^d$ that satisfies a local central limit theorem (LCLT). This requirement implicitly forces $d\leq 2$ and requires the migration kernel $a(\cdot\,,\cdot)$ to have a finite second moment.
		\item[\rm{(3)}] The relative strength of the seed-banks determined by $\fe$ are spatially uniformly bounded, i.e.,
		\begin{equation}
			\sup_{i\in\Z^d} \tfrac{M_i}{N_i}<\infty.
		\end{equation}
		\item[\rm{(4)}] The sizes of the active populations determined by $\fe$ are \emph{non-clumping}, i.e.,
		\begin{equation}
			\inf_{i\in\Z^d}\sum_{\|j-i\|\leq R}\tfrac{1}{N_j} > 0\quad\text{ for some } R<\infty.
		\end{equation}
	\end{itemize}
\end{assumption}
\noindent
In view of the above, unless stated otherwise, we will throughout assume that Assumption~\ref{assumpt1} and Assumption~\ref{assumpt:clustering-env} are in force. We remark that the above conditions are sufficient but not necessary for the process $\sZ$ to remain in the clustering regime.

In this exposition we refrain from introducing the dual process in full generality and only define a version of the dual consisting of a single particle in terms of a \emph{coordinate process} $\Theta^\fe$. Informally, the process $\Theta^\fe$ keeps track of the location and the state of a single dual particle in time, while the general dual $\sZ_*$ describes the evolution of the particle via configurations in $\sX_*$. The process $\Theta^\fe$ plays a key role in the proofs of all our main results, and will be our sole focus in Section~\ref{sec-single-particle-proc-in-random-env}. Later, in Section~\ref{sec-prelims-on-dual} we will explain via Lemma~\ref{lemma:relation-between-theta-dual} how the single-particle process $\Theta^\fe$ is related to the general dual process $\sZ_*$. We refer the reader to \cite[Section 3.2]{HN01} and \cite[Section 3]{HN02} for further insight into the general dual process $Z_*^\fe$.

\begin{definition}{\bf [Single-particle dual process]}
	\label{defn:single-particle-dual}
	{\rm The single-particle dual process 
		\begin{equation}
			\label{eqn:coordinate-of-single-dual}
			\Theta^\fe := (\Theta^\fe(t))_{t \geq 0}, \qquad \Theta^\fe(t) = (x^\fe_t,\alpha^\fe_t),
		\end{equation}
	in environment $\fe:=(N_i,M_i)_{i\in\Z^d}$ is the continuous-time Markov chain on the state space 
		\begin{equation}
			G:=\Z^d\times\{0,1\}
		\end{equation} 
		with transition rates
		\begin{equation}
			\label{eqn:trans-rates-of-theta}
				(i,\alpha) \to
				\begin{cases}
					\displaystyle				
					(i,1-\alpha), &\text{ at rate } 
					\lambda[\alpha+(1-\alpha)K_i],\\
					(j,\alpha), &\text{ at rate } 
					\alpha\, a(0,j-i) \quad \,\text{ for } j\neq i\in\Z^d,
				\end{cases}
		\end{equation}
		where $(i,\alpha)\in G$ and the environment $\fe$ fixes $K_i$ by \eqref{eqn:defn-of-Ki}. We define the time-$t$ probability transition kernel $p_t^\fe(\cdot\,,\cdot)\,:G\times G\to[0,1]$ associated to $\Theta^\fe$ as
		\begin{equation}
			p_t^\fe(\eta,\xi) := P_\eta^\fe(\Theta^\fe(t)=\xi),\quad\eta,\xi\in G,
		\end{equation}
	where $P^\fe_\eta$ is the law of the process $\Theta^\fe$ started at $\eta\in G$. 
	}\hfill $\blacklozenge$
\end{definition}

\noindent
The coordinates $x_t^\fe$ and $\alpha_t^\fe$ in \eqref{eqn:coordinate-of-single-dual} represent, respectively, the location in $\Z^d$ and the state (active or dormant) of the particle at time $t$, where $0$ stands for dormant and $1$ stands for active. Note from \eqref{eqn:trans-rates-of-theta} that only the wake-up rate of the particle depends on the environment $\fe$, and only via the ratios $(K_i)_{i\in\Z^d}$ defined in \eqref{eqn:defn-of-Ki}. Indeed, the average time spent in the dormant state by the particle at site $i$ is proportional to $K_i^{-1}$, the relative strength of the seed-bank at colony $i$. The particle in the active state migrates according to the kernel $a(\cdot\,,\cdot)$, and so migration is not affected by the environment $\fe$, at least not in a direct manner. This makes the analysis of the single-particle process $\Theta^\fe$ in a typical \emph{random} environment $\fe$ easier than the full dual process $Z_*^\fe$. 

\medskip\noindent
Let us now state the main result of this section.

\begin{theorem}{\bf[Domain of attraction]}
	\label{thm:domain-of-attr-cluster}
	Suppose that the process $\sZ:=(\sZ(t))_{t\geq 0}$ {\corrected exhibits clustering in the sense of Definition~{\rm\ref{defn:clustering-coexistence}}} and $\sZ(0)=(X_i^\fe(0),Y_i^\fe(0))_{i\in\Z^d}$ has distribution $\mu^\fe\in\mathcal{P}(\sX)$, where $\fe:=(N_i,M_i)_{i\in\Z^d}\in\cA$ is an arbitrarily fixed environment. If $\mu_t^\fe$ denotes the time-$t$ distribution of the process $\sZ$, then the following are equivalent:
	\begin{enumerate}[{\rm(a)}]
		\item $\mu_t^\fe$ converges weakly as $t\to\infty$.
		\item For any $(i,\alpha)\in G:=\Z^d\times\{0,1\}$, 
		\begin{equation}
			\label{eqn:limiting-behaviour-of-kernel}
			f^\fe(i,\alpha):=\lim\limits_{t\to\infty} \sum_{(j,\beta)\in G}p_t^\fe((i,\alpha),(j,\beta))\, \E_{\mu^\fe}\Big[\beta\,\tfrac{X^\fe_j(0)}{N_j}+(1-\beta)\,\tfrac{Y^\fe_j(0)}{M_j}\Big] \text{ exists},
		\end{equation}
		where $p_t^\fe(\cdot\,,\cdot)$ is as in Definition~{\rm \ref{defn:single-particle-dual}}.
	\end{enumerate}
	Further, if any of the above two conditions is satisfied, then there exists $\theta_\fe\in[0,1]$ such that $f^\fe(\cdot)\equiv\theta_\fe$ and
	\begin{equation}
		\lim\limits_{t\to\infty}\mu_t^\fe = (1-\theta_\fe)\delta_{\mathbf{0}}+\theta_\fe\delta_{\fe}.
	\end{equation} 
\end{theorem}

The following corollary states that if the process $\sZ$ exhibits clustering and starts from an initial distribution that puts a constant density of type $\heartsuit$ individuals at \emph{infinity}, then with probability 1 the spatial process $\sZ$ converges towards a mono-type equilibrium. Further, the probability of fixation to the all type-$\heartsuit$ configuration in the attained equilibrium is given by the initial density of type $\heartsuit$ in the populations at \emph{infinity}.
\begin{corollary}
	\label{coro:constant-density-at-infinity}
	Suppose that the process $\sZ$ {\corrected exhibits clustering in the sense of Definition~{\rm \ref{defn:clustering-coexistence}}} and $\mu_t^\fe$ denotes the time-$t$ distribution of the process, where $\fe:=(N_i,M_i)_{i\in\Z^d}\in\cA$ is fixed arbitrarily. If the initial distribution $\mu^\fe:=\mu_0^\fe$ is such, that for some $\theta_\fe\in[0,1]$,
	\begin{equation}
		\label{eqn:homogeneous-density-at-infinity}
		\lim\limits_{\corrected \|i\|\to\infty}\int_{\sX} \tfrac{X_i}{N_i}\,\d\mu^\fe\{(X_k,Y_k)_{k\in\Z^d}\} = \lim\limits_{\corrected \|i\|\to\infty}\int_{\sX} \tfrac{Y_i}{M_i}\,\d\mu^\fe\{(X_k,Y_k)_{k\in\Z^d}\} = \theta_\fe,
	\end{equation}
	then
	\begin{equation}
		\lim\limits_{t\to\infty}\mu_t^\fe = (1-\theta_\fe)\delta_{\mathbf{0}}+\theta_\fe\delta_{\fe}.
	\end{equation}
\end{corollary}

\noindent
The following corollary is immediate.
\begin{corollary}
\label{coro:concluding-corollary-on-domain-of-attraction}
Suppose that Assumption~{\rm \ref{assumpt1}} and Assumption~{\rm\ref{assumpt:clustering-env}} are in force. {\corrected Then, the process $\sZ$ exhibits clustering, and consequently, the results in Theorem~{\rm \ref{thm:domain-of-attr-cluster}} and Corollary~{\rm\ref{coro:constant-density-at-infinity}} hold.}
\end{corollary}

\subsection{Clustering in random environment}
\label{sec-clustering-in-random-env}
In this section we consider the process $\sZ$ in a static random environment $\fe$. Let us introduce the necessary notations before we present our main theorems. To simplify our analysis, we only consider \emph{uniformly elliptic} environments.
\begin{definition}{\bf[Uniformly elliptic environment]}
	\label{defn:elliptic-environment}
	{\rm An environment $\fe:=(N_i,M_i)_{i\in\Z^d} \in (\N^2)^{\Z^d}$ is said to be \emph{uniformly elliptic} if	\begin{equation}
		\label{eqn:elliptic-environment}
	(N_i,M_i)\in \{2,3,\ldots,\fK\}^2
	\end{equation}
	for all $i\in\Z^d$ and some natural number $\fK\geq 2$. The set of all environments satisfying \eqref{eqn:elliptic-environment} is denoted by $\cE_\fK$.}\hfill$\blacklozenge$
\end{definition} 

\noindent
From here onwards we fix a natural number $\fK\geq 2$, which we refer to as the ellipticity constant. We equip $\cE_\fK$ with the product topology and the Borel $\sigma$-field $\Sigma$. The product topology is naturally induced by the metric $\mathcal{H}\,:\cE_\fK\times\cE_\fK\to[0,\infty),$
	\begin{equation}
		\mathcal{H}((N_i,M_i)_{i\in\Z^d},(\widehat{N}_i,\widehat{M}_i)_{i\in\Z^d}) := \sum_{i\in\Z^d}\frac{1}{2^{\corrected\|i\|}}\big[1\wedge(|N_i-\widehat{N}_i|+|M_i-\widehat{M}_i|)\big].
	\end{equation}
In this metric topology, $\cE_\fK$ is a compact Polish space, and the Borel $\sigma$-field $\Sigma$ becomes countably generated.
{\corrected
\begin{remark}{\bf [Admissibility of uniformly elliptic environments]}
	\label{rem:admissibility-of-elliptic-environmets}
	 It does not immediately follow from Definition~\ref{defn:admissible-colony-sizes} that $\cE_\fK\subseteq\cA$, without the imposition of further moment conditions on the migration kernel $a(\cdot\,,\cdot)$. However, in view of Remark~\ref{rem:higher-moments-of-migration-kernel} and Definition~\ref{defn:elliptic-environment}, without loss of generality, we can enlarge the set of admissible environments $\cA$ to include $\cE_\fK$, and so the process $\sZ$ is well-defined for any $\fe\in\cE_\fK$. Furthermore, any $\fe\in\cE_\fK$ automatically satisfies conditions (3)--(4) in Assumption~\ref{assumpt:clustering-env}.
\end{remark}
}

\medskip\noindent
\begin{definition}{\bf[Translation operators]}
\label{defn:shift-operators}
{\rm For each $j\in\Z^d,$ the shift operator $T_j\,:\,\cE_\fK\to\cE_\fK$ is defined by the map
\begin{equation}
	\fe\mapsto T_j\fe,\quad T_j\fe := (N_{i+j},M_{i+j})_{i\in\Z^d},
\end{equation}
where $\fe:=(N_i,M_i)_{i\in\Z^d}\in\cE_\fK$. The action of $T_j$ on a set is interpreted pointwise, i.e., for $A\subset\cE_\fK$,  $T_jA := \{T_j\fe\,:\,\fe\in A\}$.}\hfill$\blacklozenge$
\end{definition}

\medskip\noindent
We impose the following assumption on the law of the random environment:
\begin{assumption}{\bf[Translation-invariant and ergodic field]}
	\label{assump:environment-law}
	The probability law $\bP$ of the random environment $\fe$ is defined on the measurable Polish space $(\cE_\fK,\Sigma)$ and satisfies:
	\begin{itemize}
		\item[\rm{(1)}] For any $A\in\Sigma$ and $j\in\Z^d$, $\bP(T_j^{-1}A) = \bP(A)$.
		\item[\rm{(2)}] If $A\in\Sigma$ is such that $T_j^{-1}A = A$ for all $j\in\Z^d$, then $\bP(A)\in\{0,1\}$.
	\end{itemize}
We use $\bar{\E}$ to denote the expectation w.r.t.\ $\bP$.
\end{assumption}

\medskip\noindent
We are now ready to state the main result of this section.
\begin{theorem}{\bf [Convergence in random environment]}
	\label{thm:random-env-clustering}
	Let $f_A,f_D\,:\cE_\fK\to [0,1]$ be two $\Sigma$-measurable functions such that, for $\bP$-almost every realization of $\fe:=(N_i,M_i)_{i\in\Z^d}$, the initial law $\mu^\fe\in\mathcal{P}(\sX)$ of the process $\sZ$ satisfies the following for all $i\in\Z^d$:
	\begin{equation}
		\label{eqn:consistent-initial-condition}
			\int_{\sX} \tfrac{X_i}{N_i}\,\d\mu^\fe\{(X_k,Y_k)_{k\in\Z^d}\} = f_A(T_i\fe),\quad
			\int_{\sX}\tfrac{Y_i}{M_i}\,\d\mu^\fe\{(X_k,Y_k)_{k\in\Z^d}\} = f_D(T_i\fe).
	\end{equation}
If Assumption~{\rm \ref{assumpt1}} and conditions {\rm (1)--(2)} in Assumption~{\rm \ref{assumpt:clustering-env}} hold, then, for $\bP$-almost every realization of the environment $\fe$, $\sZ(t)$ converges in law to $(1-\theta)\delta_\mathbf{0} + \theta\,\delta_\fe$ as $t\to\infty$,
where the fixation probability $\theta$ to the all type-$\heartsuit$ configuration $\fe\in\sX$ does not depend on the realization of the environment and is given by
\begin{equation}
	\label{eqn:value-of-fixation-probability}
	\theta = \frac{1}{1+\rho}\int_{\cE_\fK}\big[f_A((N_k,M_k)_{k\in\Z^d})+\tfrac{M_0}{N_0}f_D((N_k,M_k)_{k\in\Z^d})\big]\,\d\bP\{(N_k,M_k)_{k\in\Z^d}\},
\end{equation} 
with $\rho := \bar{\E}\big[\tfrac{M_0}{N_0}\big]=\int_{\cE_\fK}\tfrac{M_0}{N_0}\,\d\bP\{(N_k,M_k)_{k\in\Z^d}\}$, the average relative strength of the seed-bank in each colony.
\end{theorem}

\medskip\noindent
Let us look at a simple example where the conditions in the above theorem are met.
\begin{example}{\bf [Homogenized fixation probability]}
	Fix $\kappa\in[0,1]$. Suppose that, for a typical environment $\fe:=(N_i,M_i)_{i\in\Z^d}$ drawn from the law $\bP$, the process $\sZ$ starts with the initial law $\mu^\fe\in\mathcal{P}(\sX)$ given by
	\begin{equation}
		\mu^\fe := \bigotimes_{i\in\Z^d}\text{Binomial}(N_i,\tfrac{\kappa}{N_i})\otimes \text{Uniform}([M_i]).
	\end{equation}
In other words, in the spatial system of populations with sizes $(N_i,M_i)_{i\in\Z^d}$, initially each active individual of colony $i$ independently adopts type $\heartsuit$ with probability $\tfrac{\kappa}{N_i}$, and the number of type-$\heartsuit$ dormant individuals, which is given by $Y_i^\fe(0)$, is uniformly distributed over $[M_i]=\{0,1,\ldots,M_i\}$. In this case, if we let $f_A\,:\cE_\fK\to[0,1]$ to be the map $\fe \mapsto \tfrac{\kappa}{N_0}$ and $f_D\,:\cE_\fK\to[0,1]$ to be the constant map $\fe\mapsto\tfrac{1}{2}$, then $\mu^\fe$ satisfies
\begin{equation}
\begin{aligned}
		\E_{\mu^\fe}\big[\tfrac{X_i^\fe(0)}{N_i}\big] = \tfrac{\kappa}{N_i} = f_A(T_i\fe),\quad
	\E_{\mu^\fe}\big[\tfrac{Y_i^\fe(0)}{M_i}\big] = \tfrac{1}{2} = f_D(T_i\fe),
\end{aligned}
\end{equation}
for all $i\in\Z^d$. Thus, if the migration kernel $a(\cdot\,,\cdot)$ is symmetric, recurrent and satisfies a LCLT, then by Theorem~\ref{thm:random-env-clustering} we have that, for $\bP$-almost every realization of $\fe$, the process $\sZ$ converges in law to $(1-\theta)\delta_\mathbf{0} + \theta \delta_{\fe}$, where $\theta$ is given by
\begin{equation}
	\theta = \frac{1}{1+\bar{\E}[M_0/N_0]}\left[\bar{\E}\big[\tfrac{\kappa}{N_0}\big]+\tfrac{1}{2}\bar{\E}\big[\tfrac{M_0}{N_0}\big]\right].
\end{equation}
This tells that, in the long run, the probability of fixation of the spatial population to the all type-$\heartsuit$ configuration is $\theta$ and does not depend on the realization of the environment $\fe$. Another interesting observation is that the fixation probability  $\theta$ is an annealed average of the densities of type-$\heartsuit$ individuals. Therefore, $\theta$ is a function of the average type-$\heartsuit$ densities determined by the initial distribution $\mu^\fe$ and does \emph{not} depend on any other parameters of the distribution.
\hfill$\blacklozenge$
\end{example}
The proof of Theorem~\ref{thm:random-env-clustering} relies on the analysis of the single-particle process $\Theta^\fe$ in Definition~\ref{defn:single-particle-dual} in a random environment $\fe$ drawn from the law $\bP$. In particular, at the heart of the proof lies an exploitation of the following homogenization result, whose proof is deferred to Section~\ref{sec-homogen-in-cont-time}.
\begin{theorem}{\bf [Homogenization of environment]}
	\label{thm:weak-convergence-of-contin-env}
	Let $f_A\,:\cE_\fK\to \R$ and $f_D\,:\cE_\fK\to \R$ be two bounded $\Sigma$-measurable functions. Then, under Assumption~{\rm \ref{assumpt1}} and conditions {\rm(1)--(2)} in Assumption~{\rm \ref{assumpt:clustering-env}}, for $\bP$-almost every realization of $\fe$ and any $\alpha\in\{0,1\}$,
	\begin{equation}
		\label{eqn:deterministic-limit-of-functionals}
		\lim\limits_{t\to\infty}\sum_{(j,\beta)\in G}p_t^\fe((0,\alpha),(j,\beta))\big[\beta f_A(T_j\fe)+(1-\beta)f_D(T_j\fe)\big] = \theta,
	\end{equation} 
where $p^\fe_t(\cdot\,,\cdot)$ is the time-$t$ transition kernel of the single-particle dual process $\Theta^\fe$ given in Definition~{\rm\ref{defn:single-particle-dual}}, and
\begin{equation}
	\label{eqn:value-of-homogenised-average}
	\theta := \frac{1}{1+\rho}\int_{\cE_\fK}\big[f_A((N_k,M_k)_{k\in\Z^d})+\tfrac{M_0}{N_0}f_D((N_k,M_k)_{k\in\Z^d})\big]\,\d\bP\{(N_k,M_k)_{k\in\Z^d}\},
\end{equation}
with $\rho := \bar{\E}\big[\tfrac{M_0}{N_0}\big]=\int_{\cE_\fK}\tfrac{M_0}{N_0}\,\d\bP\{(N_k,M_k)_{k\in\Z^d}\}.$
\end{theorem}
The interpretation of the above result is that, for $\bP$-almost every realization of the environment $\fe$, the law of the \emph{``environment viewed from the particle''} in the process $\Theta^\fe$ converges weakly to an invariant distribution. The precise meaning of the last statement will become clear in Section~\ref{sec-single-particle-proc-in-random-env}. Conditions (1)--(2) in Assumption~\ref{assumpt:clustering-env} play a crucial role in the proof. Theorem~\ref{thm:weak-convergence-of-contin-env} combined with Theorem~\ref{thm:domain-of-attr-cluster} enables us to prove Theorem~\ref{thm:random-env-clustering}. 

\medskip\noindent
Note that, in \eqref{eqn:deterministic-limit-of-functionals}, the process $\Theta^\fe$ is assumed to start at $(0,\alpha)\in G$. However, this does not matter, because the law of the environment is translation-invariant {\corrected and the time-$t$ probability transition kernel $p_t^\fe(\cdot\,,\cdot)$ satisfies
	\begin{equation}
		\label{eqn:change-of-environment-and-coordinate}
		p_t^{T_i\fe}((k,\alpha), (l,\beta)) = p_t^{\fe}((k+i,\alpha), (l+i,\beta))
	\end{equation}
	for any fixed environment $\fe\in\cE_\fK$, time $t\geq 0$, locations $i,k,l\in\Z^d$, and states $\alpha,\beta\in\{0,1\}$.}
Indeed, we have the following corollary:
\begin{corollary}
	\label{corollary:limit-from-any-start-point}
	Suppose that Assumption~{\rm\ref{assumpt1}} and conditions {\rm (1)--(2)} in Assumption~{\rm \ref{assumpt:clustering-env}} hold. Let $f_A,f_D$ and $\theta$ be as in Theorem~{\rm \ref{thm:weak-convergence-of-contin-env}}. Then, for $\bP$-almost every realization of $\fe$ and all $(i,\alpha)\in\Z^d\times\{0,1\}$,
	\begin{equation}
		\label{eqn:any-point-limit-of-functionals}
		\lim\limits_{t\to\infty}\sum_{(j,\beta)\in G}p_t^\fe((i,\alpha),(j,\beta))\big[\beta f_A(T_j\fe)+(1-\beta)f_D(T_j\fe)\big] = \theta,
	\end{equation}
where $p^\fe_t(\cdot\,,\cdot)$ is as in Definition~\rm {\ref{defn:single-particle-dual}}.
\end{corollary} 

\subsection{Discussion}
\label{sec-discussion}
\paragraph{Clustering in fixed environment.}
In \cite[Theorem~3.14]{HN01} we only showed convergence of the spatial process $\sZ$ to an equilibrium for a restricted class of initial distributions, namely, a product of binomials with parameters that are tuned to the environment $\fe$ and the density of type-$\heartsuit$ individuals in the populations. The main result of Section~\ref{sec-clustering-in-fixed-env}, namely, Theorem~\ref{thm:domain-of-attr-cluster}, fully characterizes the set of initial distributions for which $\sZ$ admits convergence to equilibrium. The result is valid for any admissible environment $\fe$ in which $\sZ$ exhibits clustering. The proof follows from similar arguments used in the proof of the analogous results \cite[Theorem 1.9(b)]{L} and \cite[Theorem 1.2]{Shiga80} derived, respectively, in the context of the Voter model and the Stepping Stone model (see also e.g. \cite{Casanova18,Blath19}). In \cite[Theorem~3.17]{HN01} we showed that the process $\sZ$ clusters if and only if two dual particles in $\sZ_*$ coalesce into a single particle with probability 1. We also show in Theorem~\ref{thm:coalsence-consistency} in Section~\ref{sec-prelims-on-dual} that coalescence of two dual particles with probability 1 is equivalent to coalescence of any finite number of dual particles with probability 1. This consistency property of the dual process, which is purely a consequence of the duality relation between $\sZ$ and $\sZ_*$, is far from trivial, because the dual particles interact with each other.

To summarise, the process $\sZ$ admits \emph{only} mono-type equilibria if and only if the evolution of the dual $\sZ_*$ is eventually governed by $p_t^\fe(\cdot\,,\cdot)$, the probability transition kernel of the single-particle dual $\Theta^\fe$ (recall Definition~\ref{defn:single-particle-dual}). Precisely because of this, we see in \eqref{eqn:limiting-behaviour-of-kernel} that the domain of attraction for each mono-type equilibrium of the process $\sZ$ in the clustering regime is dictated by the limiting behaviour of $p_t^\fe(\cdot\,,\cdot)$ as $t\to\infty$. On the contrary, if the process $\sZ$ is in the coexistence regime (= existence of multi-type equilibria), then the evolution of the dual $\sZ_*$ is no longer described by $p_t^\fe(\cdot\,,\cdot)$ alone, and therefore providing an answer to similar questions in the case of coexistence is challenging. In particular, because of the presence of interactions in the dual and the lack of translation-invariance of the state space $\sX$, the characterization of the domain of attraction for a multi-type equilibrium via Liggett-type conditions (see e.g. \cite[Theorem 1.9(a)]{L},\cite{GHO1}) is a highly non-trivial problem, and is closely related to the study of harmonic functions (see e.g. \cite{Shiga80-1}) of the general dual process $\sZ_*$.
\vspace{-0.2cm}
\paragraph{Clustering in random environment.}
Turning to the main result of Section~\ref{sec-clustering-in-random-env}, we see that Theorem~\ref{thm:random-env-clustering} is a homogenization statement on the convergence of the spatial system to a mono-type equilibrium in random environment. It states that if the population sizes are drawn from an ergodic and translation-invariant random field for which clustering prevails, and the initial average densities of type-$\heartsuit$ active and dormant individuals in each colony are modulated, respectively, by two global functions $f_A(\cdot)$ and $f_D(\cdot)$ of the population sizes, then the spatial system converges in law towards a mono-type equilibrium for almost all initial realizations of the sizes. In the attained equilibrium, the probability of fixation to the all type-$\heartsuit$ configuration is a weighted average of the two functions $f_A$ and $f_D$, and is independent of the chosen initial population sizes. In other words, the spatial process $Z^\fe$ undergoes homogenization, which, roughly speaking, can be viewed as a ``weak law of large numbers''.

A closer look at the proof in Section~\ref{sec-proof-of-main-theorems} will reveal that the homogenization comes, in essence, from the duality relation with the process $\Theta^\fe$ evolving in the same random environment. The homogenization in the continuous-time process $\Theta^\fe$, in turn, is inherited from a discrete-time subordinate Markov chain $\widehat{\Theta}^\fe$ (see Definition~\ref{defn:discrete-coordinate-process} in Section~\ref{sec-subordinate-environment-chain}). This $\widehat{\Theta}^\fe$ is embedded into the continuous-time process $\Theta^\fe$ and closely resembles a $d$-dimensional version of the random walk in random environment (RWRE) on a strip introduced in \cite{Bolt20} (see also \cite{Ilya19,Goldshied19,Uphill22} for similar models and further references). However, results derived in that context do not immediately carry over to our setting, because $\widehat{\Theta}^\fe$ fails to meet some basic irreducibility hypotheses (see e.g. \cite[Condition C]{Bolt20}). Nonetheless, it turns out that $\widehat{\Theta}^\fe$ is easier to analyse than the RWRE on a strip, as some of its transition probabilities are controlled by deterministic parameters that do not depend on the environment $\fe$. To be precise, the step distribution of a particle evolving via $\widehat{\Theta}^\fe$ on the $d$-dimensional strip $\Z^d\times\{0,1\}$ is a preassigned probability distribution $\hat{p}(\cdot)$ on $\Z^d$ and, in fact, is defined in terms of the migration kernel $a(\cdot\,,\cdot)$ of the spatial process $\sZ$. This simplicity of the subordinate Markov chain, which is similar to a property found in for random walk in random scenery (see e.g., \cite{Hollander88,Steif2006}), allows us to answer some of the highly sought-after questions in the literature on RWRE. In particular, we are able to identify a stationary and ergodic distribution for the environment viewed from the particle, with an explicit expression for the density w.r.t.\ the initial law, and establish a strong law of large numbers for the location of the particle (see Section~\ref{sec-stationary-env-weak-conv}). Moreover, when $\hat{p}(\cdot)$ is symmetric and recurrent ($d\leq 2$), we show that the environment process converges weakly to the \emph{reversible} stationary distribution in the \emph{quenched} setting. The latter is a very powerful result, which ultimately causes the homogenization found in the subordinate Markov chain $\widehat{\Theta}^\fe$, and later passes it on to the single-particle dual $\Theta^\fe$ as well. 

As argued before, the spatial process $\sZ$ acquires the homogenization via duality from $\Theta^\fe$. Indeed, a crucial observation will reveal that the homogenized fixation probability in \eqref{eqn:value-of-fixation-probability} is nothing else but the average of the two global functions $f_A$ and $f_D$ w.r.t.\ the invariant distribution of the environment process. The method employed in proving the quenched weak convergence of the environment process for $\widehat{\Theta}^\fe$ to the invariant distribution is not probabilistic and relies on ergodic theoretic tools. To be precise, we first show that the \emph{peripheral point-spectrum} (i.e., the set of all eigenvalues of modulus 1) of the self-adjoint Markov kernel operator $\fR$ associated to the environment process is trivial (see Lemma~\ref{lemma:peripheral-spectrum-of-kernel} in Section~\ref{sec-stationary-env-weak-conv}) and afterwards invoke a generalised version of the fundamental theorem for Markov chains (see Proposition~\ref{prop:fundamental-theorem-of-MC} in Section~\ref{sec-stationary-env-weak-conv}) to establish the convergence. This way of proving weak convergence of the environment process is non-standard in the literature on RWRE, where such convergences are often established by exploiting some form of regeneration structure, or results like a local central limit theorem for the relevant random walk (see e.g., \cite{Kesten77,Lally86,Ilya19,Birkner21}). Admittedly, the analysis of the peripheral point-spectrum of a Markov kernel operator in the $L_p$ ($p\geq 1$) space of its reversible distribution is non-trivial and requires knowledge of the explicit form of the distribution. However, in many random environment models, such as the random conductance model, the one-dimensional RWRE, etc., important results in the quenched setting are still incomplete despite the existing knowledge of their explicit reversible distributions. Perhaps such problems may be approached in a similar way.

\section[Single-particle dual in random environment]{Single-particle dual in random environment}
\label{sec-single-particle-proc-in-random-env}
As indicated in the previous section, the single-particle dual process $\Theta^\fe$ (see Definition~\ref{defn:single-particle-dual}) serves as the main ingredient in proofs of all our main results. In this section we study $\Theta^\fe$ in a typical random environment $\fe\in\cE_\fK$ drawn according to the law $\bP$ (see Assumption~\ref{assump:environment-law}) and prove the homogenization result stated in Theorem~\ref{thm:weak-convergence-of-contin-env}.

To avoid dealing with technicalities that arise in the context of continuous-time Markov processes, in Section~\ref{sec-subordinate-environment-chain} we transform the process $\Theta^\fe$ into a discrete-time Markov chain $\widehat{\Theta}^\fe$ using the well-known method of \emph{uniformization} by a Poisson clock. We also introduce an \emph{auxiliary environment process} $W$ associated to the Markov chain $\widehat{\Theta}^\fe$. In Section~\ref{sec-stationary-env-weak-conv} we show that the environment process $W$ converges weakly to an invariant distribution in the \emph{quenched} setting. Finally, in Section~\ref{sec-homogen-in-cont-time} we prove Theorem~\ref{thm:weak-convergence-of-contin-env} and Corollary~\ref{corollary:limit-from-any-start-point} by transferring the convergence result on $W$ to the continuous-time process $\Theta^\fe$.

\subsection{Subordinate Markov chain and auxiliary environment process}
\label{sec-subordinate-environment-chain}

 When a continuous-time Markov process on a countable state space retains \emph{uniformly bounded} jump rates, it can be uniformized by a Poisson clock and a discrete-time subordinate Markov chain (see e.g., \cite[Chapter 2]{Liggett10}). The method of uniformization essentially transforms a \emph{variable-speed} continuous-time Markov process into a \emph{constant-speed} continuous-time Markov process \cite{Biskup11}. Observe from \eqref{eqn:trans-rates-of-theta} that the jump rates of $\Theta^\fe$ (see Definition~\ref{defn:single-particle-dual}) are uniformly bounded when the chosen environment $\fe$ is uniformly elliptic, and therefore $\Theta^\fe$ is uniformizable for such an environment. We start by defining a subordinate Markov chain $\widehat{\Theta}^\fe$ corresponding to the process $\Theta^\fe$ in a uniformly elliptic environment $\fe$. 
\begin{definition}{\bf [Subordinate Markov chain]}
	\label{defn:discrete-coordinate-process}
	{\rm The subordinate Markov chain (see Fig.~\ref{fig:representation-of-subordinate-chain})
		\begin{equation}
			\widehat{\Theta}^\fe := (\widehat{\Theta}^\fe_n)_{n\in\N_0}, \qquad \widehat{\Theta}^\fe_n = (X^\fe_n,\alpha^\fe_n),
		\end{equation}
		in a uniformly elliptic environment $\fe:=(N_i,M_i)_{i\in\Z^d}\in\cE_\fK$ is the discrete-time Markov chain on the state space 
		$G=\Z^d\times\{0,1\}$ with transition probabilities
		\begin{equation}
			\label{eqn:trans-prob-of-aux-proc}
			\begin{aligned}
				&(i,1)\longrightarrow
				\begin{cases}
					(j,1)&\text{ w.p.\ } (1-q_s)\hat{p}(j-i),\quad j\in\Z^d,\\
					(i,0)&\text{ w.p.\ } q_s,
				\end{cases}\\
				&(i,0)\longrightarrow
				\begin{cases}
					(i,0) &\text{ w.p.\ } 	1-\omega(i),\\
					(i,1)&\text{ w.p.\ } \omega(i),
				\end{cases}
			\end{aligned}
		\end{equation} 
		where $i\in\Z^d$, and the parameters $q_s$, $\omega:=(\omega(k))_{k\in\Z^d}$ and $\hat{p}:=(\hat{p}(k))_{k\in\Z^d}$ are determined by the exchange rate $\lambda$, the environment $\fe$, the migration kernel $a(\cdot\,,\cdot)$, and the ellipticity constant $\fK\geq 2$, as follows:
		\begin{equation}
			\label{eqn:transformation-of-parameters}
			\begin{aligned}
				q_s&:= \frac{\lambda}{c+\lambda+\lambda \fK},\quad \omega(i) := \frac{\lambda K_i}{c+\lambda+\lambda \fK} = \frac{\lambda N_i}{M_i(c+\lambda+\lambda \fK)},\\ \hat{p}(i) &:= \frac{\lambda \fK}{c+\lambda \fK}\mathbf{1}_{\{i=0\}} + \frac{a(0,i)}{c+\lambda \fK}\mathbf{1}_{\{i\neq 0\}},
			\end{aligned}
		\qquad i \in\Z^d,
		\end{equation}
	where $c$ is the speed of migration defined in condition (3) of Assumption~\ref{assumpt1}. We denote by $Q_\fe(\cdot\,,\cdot)\,:G\times G\to[0,1]$ the 1-step transition kernel of the chain $\widehat{\Theta}^\fe$, defined as
	\begin{equation}
		\label{eqn:discrete-transition-kernel}
		Q_\fe(\eta,\xi) := \widehat{P}^\fe_\eta(\widehat{\Theta}^\fe_1=\xi),\quad\eta,\xi\in G,
	\end{equation}
	where $\widehat{P}^\fe_\eta$ is the canonical law of $\widehat{\Theta}^\fe$ started at $\eta$.
	}\hfill $\blacklozenge$
\end{definition}
\begin{figure}[h!t!]
	\begin{center}
		\begin{tikzpicture}[node distance=1.5cm,style={minimum size=5mm},xscale=0.8,yscale=0.8]
			\draw (0,0) rectangle (16,6);
			\pgflowlevelsynccm
			
			\draw[thick] (1.5,1) -- (14.5,1);
			\foreach \x in {2,4,...,14} {\draw[fill,black] (\x,1) circle (2pt);};
			\draw[thick] (1.5,4) -- (14.5,4);
			\foreach \x in {2,4,...,14} {\draw[fill,black] (\x,4) circle (2pt);};
			\foreach \x in {2,4,...,14} {
				\draw[dashed,thick] (\x,1) -- (\x,4);
			};
			\foreach \x in {2,4,...,14} {
				\pgfmathsetmacro\result{\x/2};
				\draw(\x,4) node[above] {$(x_{\pgfmathprintnumber{\result}},0)$};
				\draw(\x,1) node[below] {$(x_{\pgfmathprintnumber{\result}},1)$};
			};
			\draw(1.7,4) node[below] {Dormant layer};
			\draw(1.5,1) node[above] {Active layer};
			\draw (8.2,1.2) edge[->,out=15,in=165,thick] node[xshift=0.6cm,above] {$(1-q_s)\widehat{p}(x_6-x_4)$} (12,1.2);
			\draw (8.2,1.2) edge[->,out=45,in=-45,thick] node[right] {$q_s$} (8.3,3.8);
			\draw (7.8,3.8) edge[->,out=-135,in=135,thick] node[yshift=0.2cm,left] {$\omega(x_4)$} (7.8,1.2);
			\node (t-origin) at (8,4.4) {};
			\draw (t-origin) edge[->,out=30,in=150,thick,loop] node[above] {$1-\omega(x_4)$} (7.8,1.2);
		\end{tikzpicture}
		\caption{\small A schematic representation of the transition probabilities of a particle moving on the $d$-dimensional strip $\Z^d\times\{0,1\}$ according to $\widehat{\Theta}^\fe$. The particle is allowed to migrate in the bottom layer and while doing so remains in \emph{active} state. However, the particle becomes \emph{dormant} by entering the top layer, and thus can not migrate.}
		\label{fig:representation-of-subordinate-chain}
	\end{center}
\end{figure}
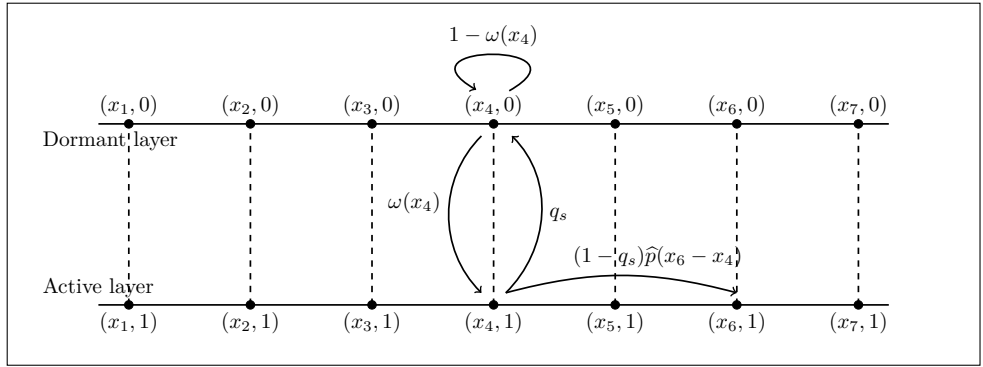
\begin{remark}{\bf[Well-posedness]}
	{Observe that $\hat{p}(\cdot)$ defines a probability distribution on $\Z^d$ and inherits the role of the migration kernel $a(0,\cdot)$. By the uniform ellipticity of the environment $\fe\in\cE_\fK$, it follows that $\omega\in[\delta,1-\delta]^{\Z^d}$ for some $\delta\in(0,\tfrac{1}{2})$ determined by $c,\lambda$ and $\fK$. Thus, the transition probabilities in \eqref{eqn:trans-prob-of-aux-proc} are well-defined. From \eqref{eqn:transformation-of-parameters} we see that $\omega$ is the only parameter that depends on $\fe$ and plays the role of \emph{random} environment for $\widehat{\Theta}^\fe$, while $q_s$ takes over the role of $\lambda$, which is the rate of becoming dormant from the active state in the continuous-time process $\Theta^\fe$.}
\end{remark}
\medskip\noindent
The subordinate Markov chain $\widehat{\Theta}^\fe$ describes the evolution of a particle moving on the $d$-dimensional strip $\Z^d\times\{0,1\}$ in discrete time. The coordinates $X_n^\fe$ and $\alpha_n^\fe$ give, respectively, the location in $\Z^\d$ and the state (active or dormant) at time $n\in\N_0$ of the particle evolving in the environment $\fe$ according to the transition probabilities given in \eqref{eqn:trans-prob-of-aux-proc}. In each step, the particle in the active state, with probability $(1-q_s)$, performs random walk on $\Z^d$ according to the increment distribution $\hat{p}(\cdot)$, while, with probability $q_s$, it becomes dormant from the active state. The particle does not move in the dormant state and becomes active with a location-dependent probability determined by the environment $\fe$. The following property of the law of $\widehat{\Theta}^\fe$ is a consequence of the translation-invariance of $\Z^d$ and the migration kernel $a(\cdot\,,\cdot)$. The proof follows from an easy calculation of the transition probabilities of $\widehat{\Theta}^\fe$ given in \eqref{eqn:trans-prob-of-aux-proc}, and is omitted for briefness.
\begin{lemma}{\bf[Translation-invariance]}
	\label{lemma:cycle-property-of-law}
	For any $(i,\alpha), (j,\beta)\in G$ and $n\in\N_0$,
 	\begin{equation}
		\widehat{P}^\fe_{(0,\alpha)}(\widehat{\Theta}^\fe_n=(j,\beta)) = \widehat{P}^{T_{-i}\fe}_{(i,\alpha)}(\widehat{\Theta}^{T_{-i}\fe}_n= (i+j,\beta)).
	\end{equation}
\end{lemma}
The connection between the discrete-time Markov chain $\widehat{\Theta}^\fe$ and the continuous-time Markov process $\Theta^\fe$ becomes apparent in the next lemma.
\begin{lemma}{\bf [Uniformization by Poisson clock]}
	\label{lemma:poisson-clock-representation}
	Let $\fe\in\cE_\fK$ be a uniformly elliptic environment and $(N_t)_{t\geq 0}$ be a Poisson process with rate $c+\lambda+\lambda \fK$ that is independent of the subordinate Markov chain $\widehat{\Theta}^\fe$. Then, under the assumption that the process $\Theta^\fe$ (see Definition~{\rm \ref{defn:single-particle-dual}}) and the Markov chain $\widehat{\Theta}^\fe$ have the same initial distribution,
	\begin{equation}
		(\Theta^\fe(t))_{t\geq 0} \overset{d}{=}(\widehat{\Theta}^\fe_{N_t})_{t\geq 0}.
	\end{equation}
In particular, for $\eta,\xi\in G$,
\begin{equation}
	p_t^\fe(\eta,\xi) = \e^{-(c+\lambda+\lambda \fK)t}\sum_{n=0}^\infty \tfrac{[(c+\lambda+\lambda \fK)t]^n}{n!}Q^n_\fe(\eta,\xi) ,
\end{equation}
where $p_t^\fe(\cdot\,,\cdot)$ and $Q_\fe(\cdot\,,\cdot)$ are as in Definition~{\rm \ref{defn:single-particle-dual}} and Definition~{\rm\ref{defn:discrete-coordinate-process}}, respectively.
\end{lemma}
\begin{proof}
	Let $\mathcal{J}_\fe$ denote the infinitesimal generator of the process $\Theta^\fe$. Since $\fe$ is uniformly elliptic, it is easily seen that $\mathcal{J}_\fe$ is a bounded operator and thus $(\exp\{\mathcal{J}_\fe t\})_{t\geq 0}$ defines the semigroup of $\Theta^\fe$. In particular, the transition probability kernel $p_t^\fe(\cdot\,,\cdot)$ expands as
	\begin{equation}
		p_t^\fe(\cdot\,,\cdot) = \sum_{n=0}^\infty \mathcal{J}_\fe^n(\cdot\,,\cdot) \tfrac{t^n}{n!},
	\end{equation}
where the generator $\mathcal{J}_\fe$ is viewed as a matrix. The claim follows from this expansion of $p_t^\fe(\cdot\,,\cdot)$ and the observation that
\begin{equation}
	\label{eqn:generator-relation}
	\mathcal{J}_\fe = (c+\lambda+\lambda \fK)[Q_\fe-I],
\end{equation}
where $I$ is the identity operator (viewed as a matrix). Note that in \eqref{eqn:generator-relation} the translation-invariance of the migration kernel $a(\cdot\,,\cdot)$ is used.
\end{proof}
Below we define the ``environment process'' associated to the subordinate Markov chain $\widehat{\Theta}^\fe$. This process is defined in the same way as for RWRE on a strip (see e.g., \cite[Definition 2.2]{Goldshied19}).

\begin{definition}{\bf [Auxiliary environment process]}
	\label{defn:environment-chain}
	{\rm Let $\widehat{\Theta}^\fe=(X_n^\fe,\alpha_n^\fe)_{n\in\N_0}$
	 with the canonical law $\widehat{P}^\fe_{(0,\alpha)}$ be the subordinate Markov chain (see Definition~\ref{defn:discrete-coordinate-process}) started at $(0,\alpha)\in G$ in environment $\fe\in\cE_\fK$. The auxiliary environment process $W$ having initial distribution $\delta_{(\fe,\alpha)}$ is the discrete-time process on $\Omega_\fK:=\cE_\fK\times\{0,1\}$ given by
	\begin{equation}
		W := (W_n)_{n\in\N_0},\quad W_n := (\fe_n,\alpha_n) \text{ with } \fe_n := T_{X^\fe_n}\fe,\ \alpha_n:=\alpha_n^\fe,
	\end{equation}
	and is defined on the same probability space of $\widehat{\Theta}^\fe$.
	}\hfill$\blacklozenge$
\end{definition}

\medskip\noindent
It is trivial to check that, for any $(\fe,\alpha)\in\Omega_\fK$, 
$W$ is a Markov chain on the state space $\Omega_\fK$
under the law $\widehat{P}^\fe_{(0,\alpha)}$, with initial distribution $\delta_{(\fe,\alpha)}$ [by Lemma~\ref{lemma:cycle-property-of-law}, also under the law $\widehat{P}^\fe_{(i,\alpha)}$, $i\in\Z^d$, with initial distribution $\delta_{(T_i\fe,\alpha)}$].

The action of the Markov kernel operator $\fR$ associated to $W$ on a bounded function $f\in\mathcal{F}_b(\Omega_\fK)$ is given by
\begin{equation}
\fR f(\fe,\alpha):=\widehat{E}^\fe_{(0,\alpha)}[f(W_1)] = \sum_{(j,\beta)\in G}Q_\fe((0,\alpha),(j,\beta))f(T_j\fe,\beta),
\end{equation}
where $(\fe,\alpha)\in\Omega_\fK$ and $Q_\fe(\cdot\,,\cdot)$ is the 1-step transition kernel of $\widehat{\Theta}^\fe$ defined in \eqref{eqn:discrete-transition-kernel}. In particular,
\begin{equation}
	\label{eqn:transition-kernel-environment-chain}
	\begin{aligned}
		\fR f(\fe,\alpha)=
		\begin{cases}
			q_s\,f(\fe,0)+(1-q_s)\displaystyle\sum_{j\in\Z^d}\hat{p}(j)f(T_j\fe,1), &\text{ if } \alpha=1,\\
			\omega(0)f(\fe,1)+[1-\omega(0)]f(\fe,0),&\text{ otherwise,}
		\end{cases}
	\end{aligned}
\end{equation}
where $q_s$, $\hat{p}(\cdot)$ and  $\omega:=(\omega(k))_{k\in\Z^d}$ are defined in terms of $\fe$ and the other parameters in \eqref{eqn:transformation-of-parameters}.

The Markov chain $W$ describes the state of the environment from the point of view of a particle that moves  on the $d$-dimensional strip $\Z^d\times\{0,1\}$ according to the chain $\widehat{\Theta}^\fe$. The definition of the process differs from the standard definition usually encountered in the literature on RWRE. This is because the particle moves on two copies of $\Z^d$ instead of one, and in order to preserve the Markov property we need an extra variable describing the layer on which the particle is present.

The state space $\Omega_\fK$ of the auxiliary environment process $W$, even though compact, is huge. Thus, at first glance, obtaining any useful information from $W$ might seem to be an impossible task. In general, this difficulty is overcome by taking initial samples of the environment from an ergodic and translation-invariant law. In such settings, it often becomes possible to construct ``by hand'' an invariant distribution that is absolutely continuous w.r.t.\ the initial law. Invariant distributions having such characteristics, which guarantees its uniqueness as well (see e.g. \cite{Sznitman02,Kozlov85}), are an extremely powerful tool for deriving many interesting properties, such as laws of large numbers, central limit theorems etc.,\ for the relevant process. In the next section we find an invariant distribution $\Q$ with such a property and prove weak convergence of $W$ to the invariant distribution in the \emph{quenched} setting.

\subsection[Stationary environment process and weak convergence]{Stationary environment process and weak convergence}
\label{sec-stationary-env-weak-conv}
In this section we address the question of whether the auxiliary environment process $W$ admits an invariant distribution that is ``equivalent'' to its initial distribution. The following result provides a positive answer:

\begin{theorem}{\bf[Invariant distribution of environment process]}
	\label{thm:invariant-distribution-of-W}
	Let $\Q$ be the probability measure on $(\Omega_\fK,\Sigma\otimes 2^{\{0,1\}})$ defined by
	\begin{equation}
		\label{eqn:defn-of-invariant-dist}
		\d\Q\{(\fe,\alpha)\} := \frac{u(\fe,\alpha)}{1+\rho}\,\d\bP\{\fe\},
	\end{equation}
where the law $\bP$ defined on $(\cE_\fK,\Sigma)$ is as in Assumption~{\rm \ref{assump:environment-law}}, $\rho:=\bar{\E}\big[\tfrac{M_0}{N_0}\big]$, and the density $u\,:\Omega_\fK\to(0,\fK]$ is given by
\begin{equation}
	u((N_k,M_k)_{k\in\Z^d},\alpha)=\begin{cases}
		1 &\text{ if } \alpha=1,\\
		\tfrac{M_0}{N_0} &\text{ if } \alpha=0.
	\end{cases}
\end{equation}
The following hold:
\begin{enumerate}[{\rm(1)}]
	\item The environment process $W$ in Definition~{\rm \ref{defn:environment-chain}} is stationary and ergodic under the probability law $\Q$.
	\item Under condition {\rm(1)} in Assumption~{\rm\ref{assumpt:clustering-env}}, $\Q$ is reversible.
\end{enumerate}
\end{theorem}
\begin{remark}{\bf [Validity in all dimensions]}
	\label{rem:validity-in-all-dimension}
{Part (1) of Theorem~\ref{thm:invariant-distribution-of-W} holds without the imposition of condition (1) in Assumption~\ref{assumpt:clustering-env}. It essentially follows from the translation-invariance and ergodicity of the law $\bP$. Moreover, both part (1) and part (2) are valid in all dimensions $d\geq 1$. Assumption~\ref{assumpt1} is crucial for the proof and can not be removed in a straightforward way.
}
\end{remark}

The proof of Theorem~\ref{thm:invariant-distribution-of-W} is mostly computational and is deferred to Appendix~\ref{sec-stationarity-n-large-numbers}. As an application of this result, in Appendix~\ref{sec-stationarity-n-large-numbers} we also give a proof of strong law of large numbers for the subordinate Markov chain $\widehat{\Theta}^\fe$ (recall Definition~\ref{defn:discrete-coordinate-process}), which is a result of independent interest.

Before we proceed further, let us explain what we mean by ``equivalence'' of the invariant distribution $\Q$ in the theorem and the initial law $\bP$ of the environment. In the literature on RWRE, this phenomenon is called ``equivalence between the static and the dynamic points of view''.
\begin{lemma}{\bf [Equivalence of $\Q$ and $\bP$]}
	\label{lemma:equivalence-of-pi-P}
	Let $\Q,\bP$ be as in Theorem~{\rm \ref{thm:invariant-distribution-of-W}}. Then, for any measurable $A \subseteq\Omega_\fK = \cE_\fK\times\{0,1\}$, the following are equivalent:
	\begin{itemize}
		\item[{\rm(1)}] $\Q(A) = 1$.
		\item[{\rm(2)}] There exists a $\Sigma$-measurable $A^\prime\subseteq \cE_\fK$ such that $\bP(A^\prime)=1$ and $A^\prime\times\{0,1\}\subseteq A$.
	\end{itemize}
\end{lemma}
\begin{proof}
	Let $\theta := \tfrac{1}{1+\bar{\E}[M_0/N_0]}\in(0,1)$, and let $\mu$ be the probability measure on $(\cE_\fK,\Sigma)$ defined by
		\begin{equation}
			\label{eqn:defn-of-mu}
			\mu(E) = \tfrac{\theta}{1-\theta}\int_{E}\tfrac{M_0}{N_0}\,\d\bP\{(N_k,M_k)_{k\in\Z^d}\},\quad E\in\Sigma.
		\end{equation}
	Clearly, for any $E\in\Sigma$,
	 \begin{equation}
	 	\label{eqn:mu-equivalency}
	 	\mu(E)=1 \text{ if and only if } \bP(E)=1.
	 \end{equation}
 Suppose that (1) holds for some measurable $A\subseteq\Omega_\fK$. Note from \eqref{eqn:defn-of-invariant-dist} that
\begin{equation}
	1=\Q(A) = \theta\,\bP(A_1) +(1-\theta)\mu(A_0),
\end{equation}
where
\begin{equation}
	A_0:=\{\fe\,:\,(\fe,0)\in A\},\quad A_1:=\{\fe\,:\,(\fe,1)\in A\}.
\end{equation}
Since $\theta\in(0,1)$, this implies $\bP(A_1)=\mu(A_0)=1$ . Defining $A^\prime = A_0\cap A_1$, we see that (2) follows from \eqref{eqn:mu-equivalency}.

Similarly, if (2) holds, then by \eqref{eqn:mu-equivalency}, $\Q(A^\prime\times\{0,1\})=\theta\,\bP(A^\prime) +(1-\theta)\mu(A^\prime)=1$. Thus, $\Q(A)\geq \Q(A^\prime\times\{0,1\})=1$ and so (1) is proved.
\end{proof}

Our next goal is to prove weak convergence of the environment process $W$ to the stationary law $\Q$ under the quenched law $\widehat{P}_{(0,\alpha)}^\fe$ for $\bP$-almost every realization of the environment $\fe\in\cE_\fK$. In particular, we have the following result:
\begin{theorem}{\bf [Weak convergence of auxiliary environment]}
	\label{thm:weak-convergence-of-discrete-env}
	Suppose that conditions {\rm (1)--(2)} in Assumption~{\rm \ref{assumpt:clustering-env}} hold. Let $f_A\,:\cE_\fK\to\R$ and $f_D\,:\cE_\fK\to\R$ be two bounded $\Sigma$-measurable functions. Then, for $\bP$-almost every realization of $\fe\in\cE_\fK$ and any $\alpha\in\{0,1\}$,
	\begin{equation}
		\label{eqn:weak-convergence-of-expectation}
		\lim\limits_{n\to\infty}\widehat{E}_{(0,\alpha)}^\fe[h(\fe_n,\alpha_n)] = \int_{\cE_\fK\times\{0,1\}} h(\fe^\prime,\beta)\,\d\Q(\fe^\prime,\beta),
	\end{equation}
where $h$ is the function $(\fe,\alpha)\mapsto \alpha f_A(\fe)+(1-\alpha)f_D(\fe)$, $W=(\fe_n,\alpha_n)_{n\in\N_0}$ is the auxiliary environment process with law $\widehat{P}_{(0,\alpha)}^\fe$ defined in Definition~{\rm\ref{defn:environment-chain}}, and $\Q$ is the stationary law of $W$ given in \eqref{eqn:defn-of-invariant-dist}.
\end{theorem}
The proof of Theorem~\ref{thm:weak-convergence-of-discrete-env} is a consequence of the proposition stated below.
This proposition is an analogue of the ``fundamental theorem of Markov chains on countable state spaces'' because it addresses Markov chains on general state spaces. We believe that this result is already known in the literature (see e.g., \cite{Lin82} or \cite{Burkholder61,Iwanik86,Cohen14}) on ergodic theory on Markov chains, but we have been unable to find a reference with an explicit proof of the statement. For the sake of completeness, the proof is given in Appendix~\ref{sec-fundamental-theorem-of-MC}.
\begin{proposition}{\bf [Fundamental theorem of MC]}
	\label{prop:fundamental-theorem-of-MC}
	 Let $(\Omega,\Sigma,\Q)$ be a probability space, where the $\sigma$-field $\Sigma$ is countably generated. Let $W := (W_n)_{n\in\N_0}$ be a Markov chain on the state space $\Omega$, and assume that $\Q$ is a reversible and ergodic stationary distribution for $W$. If $-1$ is not an eigenvalue of the Markov kernel operator $\fR\,:L_\infty(\Omega,\Q)\to L_\infty(\Omega,\Q)$ associated to $W$, then for every bounded measurable function $f \in \mathcal{F}_b(\Omega)$ and $\Q$-almost every $w\in\Omega,$
\begin{equation}
	\label{eqn:convergence-of-expectation}
	\lim\limits_{n\to\infty}\E_w[f(W_n)] = \int_{\Omega}f\,\d\Q,
\end{equation}
where the expectation on the left is taken w.r.t.\ the law of $W$ started at $w$.
\end{proposition}
\begin{remark}{\bf[Convergence in total variation]}
	{The above proposition only establishes weak convergence and gives no information on the rate of convergence in \eqref{eqn:convergence-of-expectation}. Under more stringent classical conditions on $W$, such as Harris recurrence or a Doeblin criterion (see e.g., \cite{Meyn93,Nummelin84} and \cite{Rob2004,Kulik7} for further references), uniqueness of the law $\Q$ holds and the chain converges in total variation norm from \emph{all} initial starting points. The existence of a \emph{spectral gap} of the operator $\fR$ results in \emph{geometric ergodicity}, where the convergence takes place at an exponential rate (see e.g., \cite{Kon11}). However, under the assumption of only aperiodicity and $\phi$-irreducibility of the Markov chain $W$, convergence in total variation holds only for $\Q$-almost all initial points.}
\end{remark}

\medskip\noindent
Although in the above remark we discuss convergence of a Markov chain in total variation norm, the reader should not hope for such a strong convergence of the auxiliary environment process $W$ given in Definition~\ref{defn:environment-chain}. Indeed, the process $W$ is a highly ``singular'' Markov chain living on a huge state space $\Omega_\fK$ and admits infinitely many invariant distributions (e.g., take $\bP = \delta_\fe$, where $\fe = (N,M)_{i\in\Z^d}$ is a translation-invariant environment with $(N,M)\in\N^2$, and construct $\Q$ by \eqref{eqn:defn-of-invariant-dist}). Thus, it is very unlikely for $W$ to be Harris recurrent, or to satisfy Doeblin-type conditions for that matter.
\begin{proof}[Proof of Theorem~{\rm\ref{thm:weak-convergence-of-discrete-env}}]
	By condition (1) of Assumption~\ref{assumpt:clustering-env} and Theorem~\ref{thm:invariant-distribution-of-W}, we see that $\Q$ is a reversible and ergodic distribution for the auxiliary environment process $W$. Observe from Proposition~\ref{prop:fundamental-theorem-of-MC}, if we are able to prove that $-1$ is not an eigenvalue of the Markov kernel operator $\fR\,: L_\infty(\Omega_\fK,\Q)\to L_\infty(\Omega_\fK,\Q)$ given in \eqref{eqn:transition-kernel-environment-chain}, then we can find a measurable $E\subseteq\Omega_\fK$ such that $\Q(E)=1$ and, for all $(\fe,\alpha)\in E$, \eqref{eqn:weak-convergence-of-expectation} holds for the function $h$. In particular, using Lemma~\ref{lemma:equivalence-of-pi-P} we can find a measurable $E^\prime\subset\cE_\fK$ with $\bP(E^\prime)=1$ and \eqref{eqn:weak-convergence-of-expectation} holds for all $(\fe,\alpha)\in E^\prime\times\{0,1\}$. Thus, the proof is complete once we show that $-1$ is not an eigenvalue of $\fR$ when viewed as an operator on $L_\infty(\Omega_\fK,\Q)$. We prove this in Lemma~\ref{lemma:peripheral-spectrum-of-kernel} stated below.
\end{proof}

\begin{lemma}{\bf[Trivial peripheral point-spectrum]}
	\label{lemma:peripheral-spectrum-of-kernel}
	Let $\fR$ be the Markov kernel operator (see \eqref{eqn:transition-kernel-environment-chain}) of the auxiliary environment process $W$, and $\Q$ be the invariant distribution of $W$ given in Theorem~{\rm\ref{thm:invariant-distribution-of-W}}. If condition {\rm (2)} in Assumption~{\rm \ref{assumpt:clustering-env}} holds, then $-1$ is not an eigenvalue of the kernel operator $\fR\,:L_\infty(\Omega_\fK,\Q)\to L_\infty(\Omega_\fK,\Q)$.
\end{lemma}
\begin{proof}
Let $g\in L_\infty(\Omega_\fK,\Q)$ be such that 
\begin{equation}
	\label{eqn:negative-eignevalue}
	\fR g = - g\quad\Q\text{-a.s.}
\end{equation} 
We show $g=0$ a.s. As we will see below, this will follow from condition (2) in Assumption~\ref{assumpt:clustering-env}, which ensures that the increment distribution $\hat{p}(\cdot)$ defined in terms of $a(\cdot\,,\cdot)$ in \eqref{eqn:transformation-of-parameters} does not admit any non-constant and nonnegative bounded  subharmonic function.
With this aim, let $A\subseteq\Omega_\fK$ be measurable with $\Q(A) = 1$ and such that \eqref{eqn:negative-eignevalue} holds for all $(\fe,\alpha)\in A$. Without loss of generality, we can also assume that 
\begin{equation}
	\label{eqn:bound-on-eigenfunction}
	|g(\fe,\alpha)|\leq \|g\|_\infty\quad \forall \ (\fe,\alpha)\in A.
\end{equation}
By Lemma~\ref{lemma:equivalence-of-pi-P}, there exists a measurable $A^\prime\subseteq\cE_\fK$ such that $\bP(A^\prime)=1$ and \eqref{eqn:negative-eignevalue} holds for all $(\fe,\alpha)\in A^\prime\times\{0,1\}\subseteq A$. Using \eqref{eqn:transition-kernel-environment-chain}, we compute $\fR g$ and obtain from \eqref{eqn:negative-eignevalue} that
\begin{equation}
	\label{eqn:expanded-negative-eigenvalue}
	\begin{aligned}
		g(\fe,0)&= -\big[\omega(0) g(\fe,1) + (1-\omega(0)) g(\fe,0)\big],\\
	g(\fe,1)&= -\big[q_s\, g(\fe,0)+(1-q_s)\sum_{j\in\Z^d}\hat{p}(j)g(T_j\fe,1)\big]
	\end{aligned}
,\quad\fe\in A^\prime,
\end{equation}
where, as before, $\omega,\hat{p}$ and $q_s$ are defined by \eqref{eqn:transformation-of-parameters} in terms of $\fe$ and the other parameters. Now, using the translation invariance of $\bP$, we also have
\begin{equation}
	\bP(B_\text{inv})=1,\quad B_\text{inv}:=\bigcap_{j\in\Z^d} T_j^{-1}(A^\prime)\subseteq A^\prime,
\end{equation}
where, trivially, $B_\text{inv}$ is a translation-invariant set. We get from \eqref{eqn:expanded-negative-eigenvalue} that
\begin{equation}
	\label{eqn:subharmonic-root-equation}
	\begin{aligned}
		&g(\fe,0) = -\frac{\omega(0)}{2-\omega(0)}g(\fe,1),\\
		\sum_{j\in\Z^d} &\hat{p}(j)g(T_j\fe,1) = - \Big[\frac{2-(1+q_s)\omega(0)}{(2-\omega(0))(1-q_s)}\Big]g(\fe,1),
	\end{aligned}
\end{equation} 
for all $\fe\in B_\text{inv}$. By ellipticity (see Definition~\ref{defn:elliptic-environment}) of $\fe\in B_\text{inv}$, we can find a $\delta\in(0,\tfrac{1}{2})$ such that $\delta<\omega(0)<1-\delta$ for all $\omega=(\omega(k))_{k\in\Z^d}$ determined by $\fe\in B_\text{inv}$. In particular, setting
\begin{equation}
	\label{defn:crucial-constant-C}
	C:=\frac{1}{1-q_s}\Big[1-\tfrac{1-\delta}{1+\delta}\,q_s\Big],
\end{equation}
we see that
\begin{equation}
	\frac{2-(1+q_s)\omega(0)}{(2-\omega(0))(1-q_s)} \geq C,
\end{equation}
and also $C>1$ as $\delta\in(0,\tfrac{1}{2})$.
Combining the above with \eqref{eqn:subharmonic-root-equation}, we have
\begin{equation}
	\label{eqn:main-subharmonic-equation}
		\Big|\sum_{j\in\Z^d}\hat{p}(j)g(T_j\fe,1)\Big| =\Big|\tfrac{2-(1+q_s)\omega(0)}{(2-\omega(0))(1-q_s)}\Big| |g(\fe,1)|
		\geq C |g(\fe,1)|,\quad \fe\in B_\text{inv}.
\end{equation}
Using the triangle inequality, we get
\begin{equation}
\sum_{j\in\Z^d}\hat{p}(j)|g(T_j\fe,1)|\geq C|g(\fe,1)|,\quad\fe\in B_\text{inv}.
\end{equation}
Because $B_\text{inv}$ is translation-invariant, the above implies that for any $\fe\in B_\text{inv}$ and all $i\in\Z^d$,
\begin{equation}
	\label{eqn:final-form-subharmonic-equation}
\sum_{j\in\Z^d}\hat{p}(j)|g(T_{i+j}\fe,1)|\geq C|g(T_i\fe,1)|.
\end{equation}
Since $C>1$, the above equation tells that, for a fixed $\fe\in B_\text{inv}$, the map $i\mapsto |g(T_i\fe,1)|$ is a bounded (recall \eqref{eqn:bound-on-eigenfunction}) non-negative subharmonic function for $\hat{p}(\cdot)$. Now, by condition (2) in Assumption~\ref{assumpt:clustering-env}, a random walk on $\Z^d$ with increment distribution $\hat{p}(\cdot)$ defined as in \eqref{eqn:trans-prob-of-aux-proc} is irreducible and recurrent (see e.g., \cite[Chapter 4]{Lawler2009}). Therefore, any bounded nonnegative subharmonic function of $\hat{p}(\cdot)$ on $\Z^d (d\leq 2)$ must be a constant (by an application of Doob's submartingale convergence theorem). In particular, for any $\fe\in B_\text{inv}$ and all $i\in\Z^d$,
\begin{equation}
	\label{eqn:constant-modulas}
	|g(T_i\fe,1)| = |g(\fe,1)|.
\end{equation}
Since $C>1$, the only way in which \eqref{eqn:final-form-subharmonic-equation} complies with \eqref{eqn:constant-modulas}, is when $|g(\fe,1)|=0$, so \eqref{eqn:subharmonic-root-equation} implies that $g(\fe,0)=0$ as well. Thus, $g=0$ on $B_\text{inv}\times\{0,1\}$ and, since $\bP(B_\text{inv})=1$, we see by Lemma~\ref{lemma:equivalence-of-pi-P} that $\Q(B_\text{inv}\times\{0,1\})=1$.
\end{proof}
\begin{remark}{\bf [Peripheral point-spectrum in $L_1$]}
	Using \cite[Lemma 2]{Iwanik86}, we can actually show that $-1$ is not an eigenvalue of $\fR$ in $L_1(\Omega_\fK,\Q)$ as well. But convergence of $\fR^{2n} f$ may fail as $n\to\infty$, when it is merely assumed that $f\in L_1(\Omega_\fK,\Q)$ (see e.g., \cite{Ornstein68}), and therefore Proposition~\ref{prop:fundamental-theorem-of-MC} does not hold in general for such $f$.
\end{remark}
\subsection[Transference of convergence: discrete to continuous]{Transference of convergence: discrete to continuous}
\label{sec-homogen-in-cont-time}
In this section we prove Theorem~\ref{thm:weak-convergence-of-contin-env} and Corollary~\ref{corollary:limit-from-any-start-point} by utilising the results derived in the Section~\ref{sec-stationary-env-weak-conv}.

Before we start with the proof of Theorem~\ref{thm:weak-convergence-of-contin-env}, let us briefly elaborate on its statement. In Section~\ref{sec-subordinate-environment-chain} we introduced in Definition~\ref{defn:environment-chain} the discrete-time auxiliary environment process $W$ associated to the subordinate Markov chain $\widehat{\Theta}^\fe$. We can also, in a similar fashion, extend the definition of $W$ to construct a continuous-time environment process $w := (w_t)_{t\geq 0}$ for the single-particle dual $\Theta^\fe$ (recall Definition~\ref{defn:single-particle-dual}). Indeed, we obtain the process $w$ by simply putting
\begin{equation}
w_t := (\fe_t,\alpha_t) \text{ with } \fe_t := T_{x^\fe_t}\fe,\ \alpha_t:=\alpha_t^\fe,
\end{equation}
for each $t\geq 0$, where $\Theta^\fe = (x_t^\fe,\alpha^\fe_t)_{t\geq 0}$ is as in Definition~\ref{defn:single-particle-dual}. Upon closer inspection of \eqref{eqn:homogeneous-density-at-infinity} and the definition of $w$, we see that Theorem~\ref{thm:weak-convergence-of-contin-env} basically states that
\begin{equation}
	\label{eqn:alternative-representation-of-limit}
	\lim\limits_{t\to\infty}E^\fe_{(0,\alpha)}[\alpha_tf_A(\fe_t)+(1-\alpha_t)f_D(\fe_t)] = \theta
\end{equation}
for $\bP$-almost every realization of the environment $\fe$, where $f_A,f_D$ and $\theta$ are as in the theorem. In other words, \eqref{eqn:alternative-representation-of-limit} is equivalent to saying that the process $w$ converges in distribution to the law $\Q$ given in \eqref{eqn:defn-of-invariant-dist} for $\bP$-almost every realization of $\fe\in\cE_\fK$ and any $\alpha\in\{0,1\}$.
\begin{proof}[Proof of Theorem~{\rm\ref{thm:weak-convergence-of-contin-env}}]
	From Lemma~\ref{lemma:poisson-clock-representation}, we observe that
	\begin{equation}
		p_t^\fe((0,\alpha),(j,\beta)) = \sum_{n=0}^\infty \widehat{P}_{(0,\alpha)}^\fe(\widehat{\Theta}^\fe_n=(j,\beta))\,\P(N_t=n),\quad (j,\beta)\in G,\,\fe\in \cE_\fK,\,t\geq 0,
	\end{equation}
	where $p_t^\fe(\cdot\,,\cdot)$ is as in Definition~{\rm \ref{defn:single-particle-dual}}, $\widehat{\Theta}^\fe=(\widehat{\Theta}^\fe_n)_{n\in\N_0}$ is the subordinate Markov chain with law $\widehat{P}^\fe_{(0,\alpha)}$ (see Definition~\ref{defn:discrete-coordinate-process}) and $(N_t)_{t\geq 0}$ is the Poisson process mentioned in the lemma, which is independent of $\widehat{\Theta}^\fe$. Thus, using the above, the left-hand side of \eqref{eqn:deterministic-limit-of-functionals}, which we abbreviate by $l((\fe,\alpha),t)$ for any $t\geq 0$, can be written as
	\begin{equation}
		\label{eqn:explicit-expr-of-infinite-sum}
		\begin{aligned}
			l((\fe,\alpha),t)&=\sum_{(j,\beta)\in G}\Big[\sum_{n\in\N_0} \widehat{P}_{(0,\alpha)}^\fe(\widehat{\Theta}^\fe_n=(j,\beta))\,\P(N_t=n)\Big]\big\{\beta f_A(T_j\fe)+(1-\beta)f_D(T_j\fe)\big\}\\
			&=\sum_{n\in\N_0}\Big[\sum_{(j,\beta)\in G} \widehat{P}_{(0,\alpha)}^\fe(W_n=(T_j\fe,\beta))\,\big\{\beta f_A(T_j\fe)+(1-\beta)f_D(T_j\fe)\big\}\Big]\P(N_t=n)\\
			&=\sum_{n\in\N_0} \widehat{E}_{(0,\alpha)}^\fe\big[h(W_n)\big]\P(N_t=n),
		\end{aligned}
	\end{equation}
where the interchange of the order of summation in the second equality is justified by Fubini's theorem, $(W_n)_{n\in\N_0}$ is the auxiliary environment process (see Definition~\ref{defn:environment-chain}), and $h\,:\cE_\fK\times\{0,1\}\to\R$ is the map $(\fe,\alpha)\mapsto \alpha f_A(\fe)+(1-\alpha)f_D(\fe)$. By virtue of Theorem~\ref{thm:weak-convergence-of-discrete-env}, we can find a measurable $B\in\Sigma$ with $\bP(B)=1$ such that, for all $\fe\in B$ and any $\alpha\in\{0,1\}$,
\begin{equation}
	\lim\limits_{n\to\infty} \widehat{E}_{(0,\alpha)}^\fe\big[h(W_n)\big] = \int_{\Omega_\fK} h(\mathfrak{b},\beta)\,\d\Q(\mathfrak{b},\beta) = \theta,
\end{equation}
where $\theta$ is as in \eqref{eqn:value-of-homogenised-average}. Fix $\fe\in B$, $\alpha\in\{0,1\}$ and $\epsilon > 0$. By virtue of the above, we can find $N_\fe\in\N$ such that, for all $n\geq N_\fe$, $|\widehat{E}_{(0,\alpha)}^\fe\big[h(W_n)\big]-\theta|<\epsilon$. Finally, from \eqref{eqn:explicit-expr-of-infinite-sum}, we get
\begin{equation}
\begin{aligned}
|l((\fe,\alpha),t)-\theta|&\leq\sum_{n=0}^\infty \big|\widehat{E}_{(0,\alpha)}^\fe\big[h(W_n)\big]-\theta\big|\,\P(N_t=n)\\&\leq 2\|h\|_\infty\,\P(N_t<N_\fe)+\epsilon\,\P(N_t\geq N_\fe)\\
&\leq 2\|h\|_\infty\,\P(N_t<N_\fe)+\epsilon.
\end{aligned}
\end{equation}
Since $N_t\to\infty$ with probability 1 as $t\to\infty$, letting $t\to\infty$ in the above, we see
\begin{equation}
\limsup_{t\to\infty}|l((\fe,\alpha),t)-\theta|\leq \epsilon.
\end{equation}
As $\epsilon>0$ is arbitrary, we get that \begin{equation}
	\lim\limits_{t\to\infty} l((\fe,\alpha),t)=\theta
\end{equation}
for all $\fe\in B$ and $\alpha\in\{0,1\}$. This proves the claim in \eqref{eqn:deterministic-limit-of-functionals}.
\end{proof}
\medskip\noindent

\begin{proof}[Proof of Corollary~{\rm\ref{corollary:limit-from-any-start-point}}]
	The proof basically follows from the translation-invariance of $\bP$ and Lemma~\ref{lemma:cycle-property-of-law}. Indeed, using Theorem~\ref{thm:weak-convergence-of-contin-env}, we can find a measurable $B\in\Sigma$ such that $\bP(B)=1$ and, for all $\fe\in B$, $\alpha\in\{0,1\}$,
\begin{equation}
	\lim\limits_{t\to\infty}\sum_{(j,\beta)\in G}p_t^\fe((0,\alpha),(j,\beta))\big[\beta f_A(T_j\fe)+(1-\beta)f_D(T_j\fe)\big] = \theta,
\end{equation}
where $\theta$ is as in \eqref{eqn:value-of-homogenised-average}. Letting $B_\text{inv}:=\cap_{j\in\Z^d}T_j^{-1}B$,
we see that $B_\text{inv}\in\Sigma$ is translation-invariant and $\bP(B_\text{inv})=1$. In particular, for any $\fe\in B_\text{inv}$ and all $(i,\alpha)\in\Z^d\times\{0,1\}$,
\begin{equation}
\lim\limits_{t\to\infty}\sum_{(j,\beta)\in G}p_t^{T_i\fe}((0,\alpha),(j,\beta))\big[\beta f_A(T_j(T_i\fe))+(1-\beta)f_D(T_j(T_i\fe))\big] = \theta.
\end{equation}
Also, using Lemma~\ref{lemma:cycle-property-of-law}--\ref{lemma:poisson-clock-representation}, we see that, for any $t\geq 0$ and $(j,\beta)\in\Z^d\times\{0,1\}$,
\begin{equation}
	p^{T_i\fe}_t((0,\alpha),(j,\beta)) = p^{\fe}_t((i,\alpha),(i+j,\beta)),\quad\forall i\in\Z^d,\alpha\in\{0,1\}.
\end{equation}
Combining the last two equations, for all $(i,\alpha)\in\Z^d\times\{0,1\}$, we get
\begin{equation}
\lim\limits_{t\to\infty}\sum_{(j,\beta)\in G}p_t^{\fe}((i,\alpha),(i+j,\beta))\big[\beta f_A(T_{i+j}\fe)+(1-\beta)f_D(T_{i+j}\fe)\big] = \theta,
\end{equation}
which after a change of variable in the summation translates to
\begin{equation}
	\lim\limits_{t\to\infty}\sum_{(j,\beta)\in G}p_t^{\fe}((i,\alpha),(j,\beta))\big[\beta f_A(T_j\fe)+(1-\beta)f_D(T_j\fe))\big] = \theta.
\end{equation}
The proof is complete by the observation that $\bP(B_\text{inv})=1$, and the above holds for any $\fe\in B_\text{inv}$.
\end{proof}
\section{Proof of main theorems}
\label{sec-proof-of-main-theorem}
In this section we prove the two main results given in Section~\ref{sec-clustering-in-fixed-env}--\ref{sec-clustering-in-random-env}. In Section~\ref{sec-prelims-on-dual}, we derive a consistency property of the general dual $\sZ_*$ of the process $\sZ$. Using this preliminary result on the dual, in Section~\ref{sec-proof-of-main-theorems} we prove Theorem~\ref{thm:domain-of-attr-cluster}, Corollary~\ref{coro:constant-density-at-infinity}, and using Theorem~\ref{thm:domain-of-attr-cluster} and the previous homogenization result on the single-particle dual $\Theta^\fe$ (see Definition~\ref{defn:discrete-coordinate-process}), we prove  Theorem~\ref{thm:random-env-clustering}.
\subsection{Preliminaries: consistency of dual process}
\label{sec-prelims-on-dual}
 We start by recalling from \cite{HN01} the duality relation between the spatial process $\sZ$ and the dual process $Z_*^\fe$ that will be needed for the proof of our main theorems.
\begin{theorem}{\bf [Duality relation]{\rm \cite[Corollary 3.11]{HN01}}}
	\label{thm:duality-relation}
	Suppose that Assumption~{\rm \ref{assumpt1}} is in force. Then, for every admissible environment $\fe=(N_i,M_i)_{i\in\Z^d}\in\cA$, the following duality relation holds between the two processes $\sZ$ and $Z^\fe_*$:
	\begin{equation}
		\label{eqn:duality-relation}
		\E_U[D^\fe(\sZ(t),V)] = \E_*^{V}[D^\fe(U,Z_*^\fe(t))],\quad t\geq 0.
	\end{equation}
	Here the expectation on the left (right) side is taken w.r.t.\ the law of $\sZ$ ($Z_*^\fe$) started at $U \in\sX$ ($V \in\sX_*$), and $D^\fe:\,\sX\times\sX_*\to[0,1]$ is the duality function defined by
	\begin{equation}
		\label{eqn:defn-duality-function}
		D^\fe(U,V) = \prod_{i\in\Z^d}\frac{\binom{X_i}{n_i}}{\binom{N_i}{n_i}}\frac{\binom{Y_i}{m_i}}{\binom{M_i}{m_i}}\mathbf{1}_{n_i\leq X_i,m_i\leq Y_i},
	\end{equation}  
	with $U=(X_i,Y_i)_{i\in\Z^d}\in\sX$ and $V=(n_i,m_i)_{i\in\Z^d}\in \sX_*$.
\end{theorem}
The next lemma establishes the relation between the process $\Theta^\fe$ and the general dual $\sZ_*$. We omit the proof for brevity, as this easily follows from the fact that any injective transformation preserves the Markov property and a unique such transformation exists that maps $\Theta^\fe$ to the dual process $\sZ_*$ started at a configuration consisting of only a single particle.
\begin{lemma}{\bf[Relation between $\Theta^\fe$ and $\sZ_*$]}
	\label{lemma:relation-between-theta-dual}
	For $i\in\Z^d$, let $\vec{\delta}_{i,A}$ ({\corrected resp.}\ $\vec{\delta}_{i,D}$) $\in\sX_*$ denote the configuration containing a single active (resp.\ dormant) particle at location $i$. Formally,
	\begin{equation}
		\label{kron_del}
		\vec{\delta}_{i,A} := (\mathbf{1}_{\{n=i\}},0)_{n\in\Z^d},\quad
		\vec{\delta}_{i,D}:=(0,\mathbf{1}_{\{n=i\}})_{n\in\Z^d},
	\end{equation}
	and for $\eta=(i,\alpha)\in \Z^d\times\{0,1\}$, let $\vec{\delta}_{\eta}:= \mathbf{1}_{\alpha=1}\,\vec{\delta}_{i,A}+\mathbf{1}_{\alpha=0}\,\vec{\delta}_{i,D}$. If $\P^\varphi_\fe$ denotes the law of $\sZ_*$ started at $\varphi\in\sX_*$, then, for all $t\geq 0,$
	\begin{equation}
		p_t^\fe(\eta,\xi) = \P^{\vec{\delta}_\eta}_\fe(\sZ_*(t)=\vec{\delta}_\xi),\quad \eta,\xi\in \Z^d\times\{0,1\},
	\end{equation}
	where $p_t^\fe(\cdot\,,\cdot)$ is as in Definition~{\rm \ref{defn:single-particle-dual}}.
\end{lemma}
The following lemma, which is essentially a consequence of Assumption~\ref{assumpt1}, tells us that any bounded harmonic function of the single-particle dual process $\Theta^\fe$ is a constant. 
\begin{lemma}{\bf [Constant harmonics]}
	\label{lemma:contant-harmonic-single-particle}
	Let $\Theta^\fe = (\Theta^\fe(t))_{t\geq 0}$ be the process defined in Definition~{\rm \ref{defn:single-particle-dual}} started at $\eta\in G$ with law $P_\eta^\fe$, where $G=\Z^d\times\{0,1\}$ and $\fe:=(N_i,M_i)_{i\in\Z^d}$. Let $f:\,G\to\R$ be a bounded harmonic function for $P_\eta^\fe$, i.e.,
	\begin{equation}
		E_\eta^\fe[f(\Theta^\fe(t))] = f(\eta)\quad\text{ for all }\eta\in G,\, t\geq 0.
	\end{equation}
	Then $f$ is constant.
\end{lemma}
\begin{proof}
	Let {\corrected$\mathcal{J}_\fe$} be the infinitesimal generator of the process $\Theta^\fe$. The action of {\corrected$\mathcal{J}_\fe$} on $f$ can be written in the following concise expression:
	\begin{equation}
		({\corrected \mathcal{J}_\fe} f)(i,\alpha):= (\alpha\lambda +(1-\alpha)\lambda K_i) [f(i,1-\alpha)-f(i,\alpha)]
		+\alpha\sum_{j\in\Z^d} a(i,j)[f(j,\alpha)-f(i,\alpha)],
	\end{equation}
	where $(i,\alpha)\in G$. Since $f$ is harmonic, $({\corrected\mathcal{J}_\fe} f)\equiv 0$ and, using the above, we have $f(i,\alpha) = f(i,1-\alpha)$ for all $(i,\alpha)\in G$, which in turn implies that the function $i\mapsto f(i,1)$ is harmonic for $a(\cdot\,,\cdot)$. Applying the Choquet-Deny theorem to the irreducible and translation-invariant kernel $a(\cdot\,,\cdot)$, we get the result.
\end{proof}
By using the duality relation stated in Theorem~\ref{thm:duality-relation} and exploiting the clustering criterion given in \cite[Theorem~3.17]{HN01}, we obtain that coalescence of two dual particles with probability 1 is equivalent to coalescence of any number of dual particles with probability 1.
\begin{theorem}{\bf [Lineage consistency]}
	\label{thm:coalsence-consistency}
	Let $\P^\varphi_\fe$ denote the law of the dual process $\sZ_*$ started at $\varphi:=(n_i,m_i)_{i\in\Z^d}\in\sX_*$ and evolving in environment $\fe:=(N_i,M_i)_{i\in\Z^d}$. Let $\tau$ be first time when all particles have coalesced into a single particle in the dual process, i.e.,
	\begin{equation}
		\label{eqn:defn-coalescence-time}
		\tau := \inf\{t\geq 0\,:\,|\sZ_*(t)|=1\},
	\end{equation}
	where $\displaystyle|\varphi|:=\sum_{i\in\Z^d}(n_i+m_i)$ is the total number of initial dual particles. Then the following are equivalent:
	\begin{enumerate}[\rm(a)]
		\item $\P^\varphi_\fe(\tau<\infty)=1$ for all $\varphi\in\sX_*$ with $|\varphi|=2$.
		\item $\P^\varsigma_\fe(\tau<\infty)=1$ for all $\varsigma\in\sX_*$ with $|\varsigma|\geq 2$.
	\end{enumerate} 
\end{theorem}
\begin{proof}
	By irreducibility of the dual process $\sZ_*$, it suffices to prove the equivalence of the two statements for fixed $\varphi,\varsigma\in\sX_*$ such that $|\varphi|=2$ and $n:=|\varsigma|\geq 2$. If $n=2$, then there is nothing to prove. So assume that $n>2$. It is straightforward to see from irreducibility and the Markov property of $\sZ_*$ that if $\P^\varphi_\fe(\tau=\infty)>0$, then $\P^\varsigma_\fe(\tau=\infty)\geq \P^\varsigma_\fe(Z^*(t)=\varphi)\P^\varphi_\fe(\tau=\infty)>0$. Hence (b) implies (a).
	
	To prove that (a) implies (b), assume $\P^\varphi_\fe(\tau<\infty)=1$ and, for $t\geq 0$, set $I_t:=|\sZ_*(t)|$. Note that, since $\sZ_*$ is a coalescent process, $I_t$ is an integer-valued bounded random variable that is {\corrected non-increasing} in $t$ a.s. Thus, $I:=\lim\limits_{t\to\infty} I_t$ exists a.s.\ and  it is enough to prove that $I=1$ a.s. To this purpose, {\corrected let} $\theta\in(0,1)$ be fixed arbitrarily, and let $\sZ$ be the spatial process started at the initial distribution $\mu_\theta^\fe$ given by
	\begin{equation}
		\mu_\theta^\fe := \bigotimes_{i\in\Z^d}\text{Binomial}(N_i,\theta)\otimes\text{Binomial}(M_i,\theta).
	\end{equation}
	By \cite[Theorem 3.14]{HN01}, the process $\sZ$ converges to an equilibrium $\nu_\theta$. Also, by our assumption that $\P^\varphi_\fe(\tau<\infty)=1$ and \cite[Theorem 2.3]{HN01}, we have
	\begin{equation}
		\label{eqn:trivial-equilibrium}
		\nu_\theta = \theta \delta_{\fe}+(1-\theta) \delta_{\mathbf{0}}.
	\end{equation}
	Furthermore, if $D^\fe(\cdot\,,\cdot)$ is the duality function in \eqref{eqn:defn-duality-function}, then combining \cite[Theorem 3.14]{HN01} and the above we get
	\begin{equation}
		\begin{aligned}
			\theta &= \E_{\nu_\theta}\big[D^\fe(\sZ(0),\varsigma)\big]=\lim\limits_{t\to\infty}\E^{\varsigma}_\fe\big[\theta^{I_t}\big]= \E^{\varsigma}_\fe\big[\theta^I\big]\quad(\text{bounded convergence}),
		\end{aligned}
	\end{equation}
	which implies that $\E^\varsigma_\fe\big[\theta(1-\theta^{I-1})\big]=0$. Since $\theta\in(0,1)$, we have that $I=1$ almost surely.
\end{proof}

\subsection{Proofs: clustering in fixed and random environment}
\label{sec-proof-of-main-theorems}
We are now ready to prove the two main theorems.
\begin{proof}[Proof of Theorem~{\rm\ref{thm:domain-of-attr-cluster}}]
	To show that (a) implies (b), suppose that $\mu_t^\fe$ converges weakly to $\nu\in\mathcal{P}(\sX)$ as $t\to\infty$. Let $\theta_\fe:=\E_\nu\big[\tfrac{X_0^\fe(0)}{N_0}\big]\in[0,1]$ be fixed. Since the system is in the clustering regime by assumption, $\delta_{\mathbf{0}}$ and $\delta_{\fe}$ are the only two extremal equilibria for the process $\sZ$. Hence, we must have that \begin{equation}
		\label{eqn:convergence-to-clustering}
		\nu=(1-\theta_\fe)\delta_{\mathbf{0}}+\theta_\fe\delta_{\fe}.
	\end{equation}
	We show that $f\equiv\theta_\fe$, which will settle (b) along with the last statement of the theorem. To this end, for each $t\geq 0$, let $f_t\,:G\to[0,1]$ be defined as
	\begin{equation}
		\label{eqn:defn-of-time-t-density}
		f_t(\eta) := \sum_{(j,\beta)\in G}p_t^\fe(\eta,(j,\beta))\int_{\sX}\big[\beta\tfrac{X_j}{N_j}+(1-\beta)\tfrac{Y_j}{M_j}\big]\,\d\mu^\fe\{(X_k,Y_k)_{k\in\Z^d}\},\quad\eta\in G.
	\end{equation} 
	Let $\eta=(i,\alpha)\in G$ be arbitrary, and let $\sZ_*:=(\sZ_*(t))_{t
		\geq 0}$ be the dual process started at $\vec{\delta}_{\eta}:= \mathbf{1}_{\alpha=1}\,\vec{\delta}_{i,A}+\mathbf{1}_{\alpha=0}\,\vec{\delta}_{i,D}$, where for each $i\in\Z^d$ the configurations $\vec{\delta}_{i,A}, \vec{\delta}_{i,D}\in\sX_*$ are defined as in \eqref{kron_del}. In other words, $\vec{\delta}_{\eta}$ is the configuration with a single dual particle located at $i\in\Z^d$ with state $\alpha$. Recall from Definition~\ref{defn:single-particle-dual} that the time-$t$ transition kernel $p_t^\fe(\cdot\,,\cdot)$ of the single-particle dual process $\Theta^\fe$ is defined as
	\begin{equation}
		p_t^\fe(\eta,\zeta) := P^\fe_\eta(\Theta^\fe(t)=\zeta),\quad \eta,\zeta\in G.
	\end{equation}
	Using Lemma~\ref{lemma:relation-between-theta-dual} and appealing to the monotone convergence theorem, we get from \eqref{eqn:defn-of-time-t-density} that 
	\begin{equation}
		\label{eqn:integral-form-time-t-density}
		f_t(\eta) = \int_{\sX}\E^{\vec{\delta}_\eta}_\fe\big[D^\fe(z, \sZ_*(t))\big]\,\d\mu^\fe\{z\},
	\end{equation}
	where the expectation is w.r.t.\ the law of the dual process $\sZ_*$, and $D^\fe(\cdot\,,\cdot)$ is the duality function in \eqref{eqn:defn-duality-function}. Furthermore, applying the duality relation between $\sZ$ and $\sZ_*$ to the above identity, we get
	\begin{equation}
		f_t(\eta) = \E_{\mu^\fe}\big[D^\fe(\sZ(t),\vec{\delta}_\eta)\big] = \int_{\sX}D^\fe(z,\vec{\delta}_\eta)\,\d\mu_t^\fe\{z\}.
	\end{equation}
	However, since $\mu_t^\fe\overset{weak}{\longrightarrow}\nu$ as $t\to\infty$ {\corrected and the map $z\mapsto D^\fe(z,\vec{\delta}_\eta)$ is bounded}, combining the above with \eqref{eqn:convergence-to-clustering}, we see that
	\begin{equation}
		f(\eta)=\lim\limits_{t\to\infty}f_t(\eta) = \int_{\sX}D^\fe(z,\vec{\delta}_\eta)\,\d\nu\{z\} = \theta_\fe,
	\end{equation}
	and hence the claim is proved.
	
	To prove the converse, for $t\geq 0$, let $f_t\,: G\to[0,1]$ be as in \eqref{eqn:defn-of-time-t-density}. Applying Fubini's theorem to \eqref{eqn:integral-form-time-t-density}, for any $\eta\in G$ we have
	\begin{equation}
		f_t(\eta) = \E^{\vec{\delta}_\eta}_\fe\Big[\int_{\sX}D^\fe(z,\sZ_*(t))\,\d\mu^\fe\{z\}\Big].
	\end{equation}
	Using the Markov property of $\sZ_*$, we note that, for $t,s\geq 0$ and $\eta\in G$,
	\begin{equation}
		f_{s+t}(\eta)=\sum_{\zeta\in G}p_s^\fe(\eta,\zeta)f_t(\zeta).
	\end{equation}
	Since by assumption $f(\eta) = \lim\limits_{t\to\infty} f_t(\eta)$ exists for any $\eta\in G$, letting $t\to\infty$ in the above identity, we obtain
	\begin{equation}
		\begin{aligned}
			f(\eta) &= \lim\limits_{t\to\infty}\sum_{\zeta\in G}p_s^\fe(\eta,\zeta)f_t(\zeta)
			=\sum_{\zeta\in G}p_s^\fe(\eta,\zeta)\,\big[\lim\limits_{t\to\infty} f_t(\zeta)\big]\quad\text{(dominated convergence)}\\
			&=\sum_{\zeta\in G}p_s^\fe(\eta,\zeta)f(\zeta)= E_\eta^\fe\big[f(\Theta^\fe(s))\big].
		\end{aligned}
	\end{equation}
	Hence, in particular, $f$ is harmonic for the process $(\Theta^\fe(t))_{t\geq 0}$ and thus, by Lemma~\ref{lemma:contant-harmonic-single-particle}, $f\equiv\theta_\fe$ for some $\theta_\fe\in [0,1]$. It only remains to show that $\mu_t^\fe$ converges weakly as $t\to\infty$.
	This is equivalent to showing that, for any $\varphi\in\sX_*$,
	$\lim\limits_{t\to\infty}\E_{\mu^\fe}\big[D^\fe(\sZ(t),\varphi)\big]$ exists. Because $\mathcal{P}(\sX)$ is compact (as $\sX$ is) in the topology of weak convergence, $(\mu_t^\fe)_{t\geq 0}$ is tight. Finally, the existence of the limit ensures the convergence of the associated finite-dimensional distributions, because the family of functions $\{D^\fe(\,\cdot\,,\varphi)\,:\,\varphi\in\sX_*\}$ fixes the mixed moments of the finite-dimensional distributions of $\sZ$ (see \cite[Proposition 5.4]{HN01}), and therefore is convergence determining. Let $\varphi=(n_i,m_i)_{i\in\Z^d}\in\sX_*$ be fixed, and $\sZ_*$ be the dual process started at $\varphi$. First note that if $|\varphi|=\sum_{i\in\Z^d}(n_i+m_i)=1$, then the limit exists and equals $\theta_\fe$ by our assumption. Indeed, if $|\varphi|=1$, then $\varphi=\vec{\delta}_\zeta$ for some $\zeta\in G$. As a consequence of duality and \eqref{eqn:integral-form-time-t-density}, we see that $\E_{\mu^\fe}\big[D^\fe(\sZ(t),\varphi)\big]=f_t(\zeta)$ and hence
	\begin{equation}
		\begin{aligned}
			\lim\limits_{t\to\infty}\E_{\mu^\fe}\big[D^\fe(\sZ(t),\varphi)\big]=\lim\limits_{t\to\infty}f_t(\zeta)=f(\zeta)=\theta_\fe.
		\end{aligned}
	\end{equation}
	Now, let us fix $\varphi\in\sX_*$ such that $|\varphi|\geq 2$. Since the system is in the clustering regime, by virtue of \cite[Theorem 2.3]{HN01}, condition (a) in Theorem~\ref{thm:coalsence-consistency} is satisfied. Hence from part (b) of Theorem~\ref{thm:coalsence-consistency} it follows that $\tau <\infty$ a.s., where $\tau:=\inf\{t\geq 0\,:\,|\sZ_*(t)| = 1\}$. Using duality and the strong Markov property of the dual process, we see that
	\begin{equation}
		\begin{aligned}
			&\lim\limits_{t\to\infty}\E_{\mu^\fe}\Big[D^\fe(\sZ(t),\varphi)\Big] \overset{Fubini}{=}\lim\limits_{t\to\infty}\E^\varphi_\fe\Big[\int_{\sX}D^\fe(z,\sZ_*(t))\,\d\mu^\fe\{z\}\Big]\\
			&\qquad =\lim\limits_{t\to\infty}\E^\varphi_\fe\Big[\int_{\sX}D^\fe(z,\sZ_*(t))\,\d\mu^\fe\{z\}; \tau\leq t\Big]\\
			&\qquad\qquad\quad+\lim\limits_{t\to\infty}\underbrace{\E^\varphi_\fe\Big[\int_{\sX}D^\fe(z,\sZ_*(t))\,\d\mu^\fe\{z\}\mid \tau>t\Big]}_{\leq 1}\P^\varphi_\fe(\tau>t)\\
			&\qquad =\lim\limits_{t\to\infty}\E^{\varphi}_\fe\Big[\E^{\sZ_*(\tau)}_\fe\Big[\int_{\sX}D^\fe(z,\sZ_*(t-\tau))\,\d\mu^\fe\{z\}\Big];  \tau\leq t\Big]\\
			&\qquad=\lim\limits_{t\to\infty}\E^{\varphi}_\fe\Big[\sum_{\zeta\in G}f_{t-\tau}(\zeta)\mathbf{1}_{\{\sZ_*(\tau)=\vec{\delta}_\zeta\}};  \tau\leq t\Big],
		\end{aligned}
	\end{equation}
	where we use that the second term after the first equality converges to 0 because $\tau<\infty$ a.s., and the last equality follows from \eqref{eqn:integral-form-time-t-density} and the fact that $\sZ_*(\tau)=\vec{\delta}_\zeta$ for some $\zeta\in G$. Finally, by an application of the dominated convergence theorem, we get
	\begin{equation}
		\begin{aligned}
			&\lim\limits_{t\to\infty}\E_{\mu^\fe}\Big[D^\fe(\sZ(t),\varphi)\Big]\\
			&\qquad=\E^\varphi_\fe\Big[\sum_{\zeta\in G}\big(\lim\limits_{t\to\infty}f_{t-\tau}(\zeta)\big)\mathbf{1}_{\{\sZ_*(\tau)=\vec{\delta}_\zeta\}};  \tau<\infty\Big]\\
			&\qquad=\E^\varphi_\fe\Big[\sum_{\zeta\in G}f(\zeta)\mathbf{1}_{\{\sZ_*(\tau)=\vec{\delta}_\zeta\}};  \tau<\infty\Big]=\theta_\fe\,\P^\varphi_\fe(\tau<\infty)\quad\text{(since $f\equiv\theta_\fe$)}\\
			&\qquad=\theta_\fe.
		\end{aligned}
	\end{equation}
	This shows that there exists $\nu\in\mathcal{P}(\sX)$ such that $\mu_t^\fe$ converges weakly to $\nu$ as $t\to\infty$. Since the system clusters by assumption, we must have
	\begin{equation}
		\nu = (1-\theta_\fe)\delta_{\mathbf{0}}+\theta_\fe\delta_{\fe}.
	\end{equation}
\end{proof}

\medskip\noindent

\begin{proof}[Proof of Corollary~{\rm \ref{coro:constant-density-at-infinity}}]
	The proof basically exploits Theorem~\ref{thm:domain-of-attr-cluster} and the fact that the particle associated to the process $\Theta^\fe$ eventually leaves any finite region of the state space $G=\Z^d\times\{0,1\}$ with probability 1. It suffices to prove that condition (b) in Theorem~\ref{thm:domain-of-attr-cluster} is satisfied. Let $f\,:\Z^d\times\{0,1\}\to[0,1]$ be the map
	\begin{equation}
		f(i,\alpha):=\alpha\E_{\mu^\fe}\big[\tfrac{X_i^\fe(0)}{N_i}\big]
		+(1-\alpha)\E_{\mu^\fe}\big[\tfrac{Y_i^\fe(0)}{M_i}\big],\quad(i,\alpha)\in\Z^d\times\{0,1\},
	\end{equation}
	and let $\epsilon > 0$ be arbitrary. By \eqref{eqn:homogeneous-density-at-infinity}, there exists $N\in\N$ such that, for all $i\in\Z^d,\,{\corrected\|i\|}> N$ and $\alpha\in\{0,1\}$, $|f(i,\alpha)-\theta_\fe|<\epsilon$.
	Thus, if $p_t^\fe(\cdot\,,\cdot)$ is the time-$t$ transition kernel of the process $(\Theta^\fe(t))_{t\geq 0}$ in Definition~\ref{defn:single-particle-dual}, then for any $\eta\in G$ and $t\geq 0$,
	\begin{equation}
		\label{eqn:limit-existence-clustering-domain}
		\begin{aligned}
			&\left|\sum_{(j,\beta)\in G}p_t^\fe(\eta,(j,\beta))\Big\{\beta\,\E_{\mu^\fe}\big[\tfrac{X_j^\fe(0)}{N_j}\big]+(1-\beta)\E_{\mu^\fe}\big[\tfrac{Y_j^\fe(0)}{M_j}\big]\Big\}-\theta_\fe\right|\\
			&\qquad \leq\sum_{\stackrel{(j,\beta)\in G,}{{\corrected\|j\|}\leq N}}p_t(\eta,(j,\beta))\underbrace{\big|f(j,\beta)-\theta_\fe\big|}_{\leq 2}+\sum_{\stackrel{(j,\beta)\in G,}{{\corrected\|j\|}> N}}p_t^\fe(\eta,(j,\beta))\underbrace{\big|f(j,\beta)-\theta_\fe\big|}_{\leq \epsilon}\\
			&\qquad\leq2\,P_\eta^\fe(\Theta^\fe(t)\in \Lambda_N\times\{0,1\})+\epsilon\,P_\eta^\fe(\Theta^\fe(t)\notin \Lambda_N\times\{0,1\}),
		\end{aligned}
	\end{equation}
	where $\Lambda_N:=\Z^d\cap[0,N]^d$, and $P_\eta^\fe$ denotes the law of $(\Theta^\fe(t))_{t\geq 0}$ started at $\eta$. Since $\Lambda_N$ is finite, $\lim\limits_{t\to\infty}P_\eta^\fe(\Theta^\fe(t)\in \Lambda_N\times\{0,1\})=0$, and so letting $t\to\infty$ in \eqref{eqn:limit-existence-clustering-domain}, we get
	\begin{equation}
		\limsup\limits_{t\to\infty}\left|\sum_{(j,\beta)\in G}p_t^\fe(\eta,(j,\beta))\Big\{\beta\,\E_{\mu^\fe}\big[\tfrac{X_j^\fe(0)}{N_j}\big]+(1-\beta)\E_{\mu^\fe}\big[\tfrac{Y_j^\fe(0)}{M_j}\big]\Big\}-\theta_\fe\right|\leq \epsilon.
	\end{equation}
	As $\epsilon$ is arbitrary, we see that
	\begin{equation}
		\lim\limits_{t\to\infty}\sum_{(j,\beta)\in G}p_t^\fe(\eta,(j,\beta))f(j,\beta)=\theta_\fe
	\end{equation}
	and hence the claim follows from Theorem~\ref{thm:domain-of-attr-cluster}.
\end{proof}
\begin{proof}[Proof of Theorem~{\rm\ref{thm:random-env-clustering}}]We exploit Theorem~\ref{thm:domain-of-attr-cluster} and the homogenization result in Corollary~\ref{corollary:limit-from-any-start-point}.
We see that, because of conditions (1)--(2) in Assumption~\ref{assumpt:clustering-env} and ellipticity of the environments $\fe\in\cE_\fK$, the process $\sZ$ is in the clustering regime for every environment $\fe\in\cE_\fK$. Also, by virtue of Corollary~\ref{corollary:limit-from-any-start-point} and the assumption in \eqref{eqn:consistent-initial-condition} on initial distributions, there exists $B\in\Sigma$ such that $\bP(B)=1$, and for all $\fe\in B$ condition (b) of Theorem~\ref{thm:domain-of-attr-cluster} holds. Furthermore, we see from Corollary~\ref{corollary:limit-from-any-start-point}, that the limiting value in that condition is independent of the environment $\fe$, and is given by \eqref{eqn:value-of-fixation-probability}. Hence the result follows.
\end{proof}
\appendix
\section{Proof of stationarity and law of large numbers}
\label{sec-stationarity-n-large-numbers}
In this section we prove Theorem~\ref{thm:invariant-distribution-of-W}. As an application, we also prove a strong law of large numbers stated later in Theorem~\ref{thm:law-of-large-numbers}.
\subsection{Stationary distribution of environment process}
\label{sec-stationarity-of-aux-env-proc}
\begin{proof}[Proof of Theorem~{\rm\ref{thm:invariant-distribution-of-W}}]
	We first prove part (1) of the theorem. To prove stationarity of $W$ under $\Q$, it suffices to show that, for any bounded measurable $f\in\mathcal{F}_b(\Omega_\fK)$,
	\begin{equation}
		\int_{\Omega_\fK}\fR f(\fe,\alpha)\,\d\Q(\fe,\alpha) = \int_{\Omega_\fK}f(\fe,\alpha)\,\d\Q(\fe,\alpha),
	\end{equation}
	where $\fR$ is the Markov kernel operator given in \eqref{eqn:transition-kernel-environment-chain}. Let $\theta:=\tfrac{1}{1+\bar{\E}[M_0/N_0]}$ and $q_s$, $\hat{p}(\cdot)$, $\omega = (\omega(k))_{k\in\Z^d}$ be as in \eqref{eqn:transformation-of-parameters}, where $\omega$ is the only parameter that depends on the realization of the environment $\fe$. In terms of these parameters, from \eqref{eqn:defn-of-invariant-dist} we get that
	\begin{equation}
		\label{eqn:general-integral-of-function}
		\int_{\Omega_\fK} g(\fe,\alpha)\,\d\Q(\fe,\alpha) = \theta\int_{\Omega_\fK}\big[g(\fe,1) + \tfrac{q_s}{\omega(0)}\,g(\fe,0)\big]\,\d\bP(\fe)
	\end{equation}
for any $g\in\mathcal{F}_b(\Omega_\fK)$. Thus, taking $g=\fR f$ in the above equation, we have
	\begin{equation}
		\label{eqn:integral-of-kernel}
		\begin{aligned}
			\int_{\Omega_\fK}\fR f(\fe,\alpha)\,\d\Q(\fe,\alpha) &= \theta\int_{\cE_\fK}\big[\fR f(\fe,1) + \tfrac{q_s}{\omega(0)}\,\fR f(\fe,0)\big]\,\d\bP(\fe)= \theta(I_1+I_2),
		\end{aligned}
	\end{equation}
	where $I_1:=\int_{\cE_\fK}\fR f(\fe,1)\,\d\bP(\fe)$ and $I_2:=\int_{\cE_\fK}\tfrac{q_s}{\omega(0)}\,\fR f(\fe,0)\,\d\bP(\fe)$.
	
	Let us compute $I_1$ and $I_2$ using \eqref{eqn:transition-kernel-environment-chain}:
	\begin{equation}
		\label{eqn:integral-value-of-I-1}
		\begin{aligned}
			I_1 &= q_s\int_{\cE_\fK}f(\fe,0)\,\d\bP(\fe)+(1-q_s)\int_{\cE_\fK}\Big[\sum_{j\in\Z^d}\hat{p}(j)f(T_j\fe,1)\Big]\,\d\bP(\fe)\\
			&=q_s\int_{\cE_\fK}f(\fe,0)\,\d\bP(\fe)+(1-q_s)\sum_{j\in\Z^d}\hat{p}(j)\int_{\cE_\fK}f(T_j\fe,1)\,\d\bP(\fe)\quad\text{(bounded convergence)}\\
			&=q_s\int_{\cE_\fK}f(\fe,0)\,\d\bP(\fe)+(1-q_s)\sum_{j\in\Z^d}\hat{p}(j)\int_{\cE_\fK}f(\fe,1)\,\d\bP(\fe)\quad\text{(translation-invariance of $\bP$)}\\
			&=q_s\int_{\cE_\fK}f(\fe,0)\,\d\bP(\fe)+(1-q_s)\int_{\cE_\fK}f(\fe,1)\,\d\bP(\fe),\quad\text{\Big(using $\sum_{j\in\Z^d}\hat{p}(j)=1\Big).$}
		\end{aligned}
	\end{equation}
	Similarly,
	\begin{equation}
		\label{eqn:integral-value-of-I-2}
		\begin{aligned}
			I_2 &= \int_{\cE_\fK}\tfrac{q_s}{\omega(0)}\,\fR f(\fe,0)\,\d\bP(\fe) =q_s\int_{\cE_\fK}[f(\fe,1)-f(\fe,0)]\,\d\bP(\fe)+\int_{\cE_\fK}\tfrac{q_s}{\omega(0)}f(\fe,0)\,\d\bP(\fe).
		\end{aligned}
	\end{equation}
	Finally, adding \eqref{eqn:integral-value-of-I-1}--\eqref{eqn:integral-value-of-I-2} and using \eqref{eqn:general-integral-of-function}--\eqref{eqn:integral-of-kernel},
	we get
	\begin{equation}
		\begin{aligned}
			\int_{\Omega_\fK}\fR f(\fe,\alpha)\,\d\Q(\fe,\alpha) &= \theta(I_1+I_2)= \theta\int_{\Omega_\fK}\big[f(\fe,1) + \tfrac{q_s}{\omega(0)}\,f(\fe,0)\big]\,\d\bP(\fe)\\
			&= \int_{\Omega_\fK}f(\fe,\alpha)\,\d\Q(\fe,\alpha),
		\end{aligned}
	\end{equation}
	which proves the claim.
	
	Next we proceed to prove ergodicity of $W$ under the stationary law $\Q$. It suffices to show (see e.g. \cite{kallenberg97}) that if $A\in \Sigma\otimes2^{\{0,1\}}$ satisfies $\fR\mathbf{1}_A = \mathbf{1}_A\ \Q${\corrected-a.s.},\ then $\Q(A)\in\{0,1\}$. Thus, let us fix a measurable $A\subseteq\Omega_\fK$ such that
	\begin{equation}
		\label{eqn:harmonic-equality-of-A}
		\fR\mathbf{1}_A(\fe,\alpha) = \mathbf{1}_A(\fe,\alpha),\quad\text{ for all }(\fe,\alpha)\in B,
	\end{equation} 
	where $B\subseteq\Omega_\fK$ is measurable with $\Q(B)=1$. Define $A_0,A_1\in\Sigma$ as
	\begin{equation}
		\begin{aligned}
			\label{eqn:projection-of-set-on-layers}
			A_0 &:= \{\fe\,:\,(\fe,0)\in A\},\ A_1:=\{\fe\,:\,(\fe,1)\in A\}.
		\end{aligned}
	\end{equation}
	By Lemma~\ref{lemma:equivalence-of-pi-P}, we can find $B^\prime\in\Sigma$ such that \begin{equation}
		\label{eqn:full-support-on-both-layers}
		\bP(B^\prime)=1,\ B^\prime\times\{0,1\}\subseteq B.
	\end{equation}
	Using \eqref{eqn:transition-kernel-environment-chain}, \eqref{eqn:harmonic-equality-of-A} and \eqref{eqn:full-support-on-both-layers}, we get that, for all $\fe\in B^\prime$,
	\begin{equation}
		\begin{aligned}
			&q_s\mathbf{1}_A(\fe,0) + (1-q_s)\sum_{j\in\Z^d}\hat{p}(j)\mathbf{1}_A(T_j\fe,1) = \mathbf{1}_A(\fe,1),\\
			&\omega(0)\mathbf{1}_A(\fe,0) + (1-\omega(0))\mathbf{1}_A(\fe,1) = \mathbf{1}_A(\fe,0),
		\end{aligned}
	\end{equation}
	where $\omega$ is defined in terms of $\fe$ as in \eqref{eqn:transformation-of-parameters}. In terms of $A_0,A_1$ given in \eqref{eqn:projection-of-set-on-layers}, for all $\fe\in B^\prime$,
	\begin{equation}
		\label{eqn:expanded-harmonic-equation-of-A}
		\begin{aligned}
			&q_s\mathbf{1}_{A_0}(\fe) + (1-q_s)\sum_{j\in\Z^d}\hat{p}(j)\mathbf{1}_{A_1}(T_j\fe) = \mathbf{1}_{A_1}(\fe),\\
			&(1-\omega(0))\mathbf{1}_{A_1}(\fe) = (1-\omega(0))\mathbf{1}_{A_0}(\fe).
		\end{aligned}
	\end{equation}
	By ellipticity of $\fe\in B^\prime$, we have $\omega(0)<1$, and so the second part of the above equation implies that
	\begin{equation}
		\label{eqn:projection-have-same-mass}
		\mathbf{1}_{A_1}(\fe)=\mathbf{1}_{A_0}(\fe),\quad\fe\in B^\prime.
	\end{equation}
	Integrating the above w.r.t.\ $\bP$ over $B^\prime$ and using \eqref{eqn:full-support-on-both-layers}, we also have
	\begin{equation}
		\label{eqn:both-set-same-mass}
		\bP(A_0) = \bP(A_1).
	\end{equation}
	Note that if we show $\bP(A_1)\in\{0,1\}$, then it follows from \eqref{eqn:both-set-same-mass} that $\Q(A)\in\{0,1\}$. Indeed, from \eqref{eqn:defn-of-invariant-dist} we see that
	\begin{equation}
		\label{eqn:defn-mass-of-event}
		\begin{aligned}
			\Q(A) = \theta \Big[\bP(A_1) + \int_{A_0}\tfrac{M_0}{N_0}\,\d \bP\{(N_k,M_k)_{k\in\Z^d}\}\Big],
		\end{aligned}
	\end{equation}
	where $\theta:=\tfrac{1}{1+\bar{\E}[M_0/N_0]}$. Therefore, if $\bP(A_1)=\bP(A_0)=1$, then
	\begin{equation}
		\begin{aligned}
			\Q(A) = \theta(1 + \bar{\E}[M_0/N_0]) = 1.
		\end{aligned}
	\end{equation}
	Similarly, if $\bP(A_1)=\bP(A_0) = 0$, then by \eqref{eqn:defn-mass-of-event}, trivially $\Q(A)=0$. We prove $\bP(A_1)\in\{0,1\}$ by using ergodicity of $\bP$.
	To this purpose, let us note that \eqref{eqn:projection-have-same-mass}, combined with the first part of \eqref{eqn:expanded-harmonic-equation-of-A} and the fact {\corrected$q_s<1$}, implies
	\begin{equation}
		\label{eqn:invariance-of-A-1}
		\sum_{j\in\Z^d}\hat{p}(j)\mathbf{1}_{A_1}(T_j\fe) = \mathbf{1}_{A_1}(\fe),\quad\fe\in B^\prime.
	\end{equation}
	Define the translation invariant set $B_\text{inv} := \bigcap_{j\in\Z^d}T_j^{-1}(B^\prime)$. By translation invariance of $\bP$ we see that $\bP(B_\text{inv})=\bP(B^\prime)=1$. Also, \eqref{eqn:invariance-of-A-1} holds for all $\fe\in B_\text{inv}$. Let us fix $\fe\in B_\text{inv}$. By translation invariance of $B_\text{inv}$, we see that $T_i\fe\in B_\text{inv}$ for any $i\in\Z^d$ and so, using \eqref{eqn:invariance-of-A-1}, we get
	\begin{equation}
		\sum_{j\in\Z^d}\hat{p}(j)\mathbf{1}_{A_1}(T_jT_i\fe) = \mathbf{1}_{A_1}(T_i\fe)\implies \sum_{j\in\Z^d}\hat{p}(j-i)\mathbf{1}_{A_1}(T_j\fe) = \mathbf{1}_{A_1}(T_i\fe).
	\end{equation}
	In particular, the map $i\mapsto \mathbf{1}_{A_1}(T_i\fe)$ is harmonic for $\hat{p}(\cdot)$. Finally, because of the irreducibility of the migration kernel $a(\cdot\,,\cdot)$ (see Assumption~\ref{assumpt1}), we can apply the Choquet-Deny theorem to the $\hat{p}$-harmonic function $i\mapsto \mathbf{1}_{A_1}(T_i\fe)$ to conclude that
	\begin{equation}
		\mathbf{1}_{A_1}(T_i\fe) = \mathbf{1}_{A_1}(\fe),\quad \forall i\in\Z^d.
	\end{equation}
	In other words, $B_\text{inv}\cap A_1$ is a translation invariant subset of $\cE_\fK$, and so ergodicity of $\bP$ implies $\bP(B_\text{inv}\cap A_1)\in\{0,1\}$. But $\bP(B_\text{inv}\cap A_1)=\bP(A_1)$ because $\bP(B_\text{inv})=1$. This concludes the proof of ergodicity of $W$ w.r.t.\ the law $\Q$.
	
	It remains to prove reversibility of $\Q$ under condition (2) in Assumption~\ref{assumpt:clustering-env}. It is enough to prove that, for $f,g\in \mathcal{F}_b(\Omega_\fK)$,
	\begin{equation}
		\label{eqn:reversibility-criterion}
		\int_{\Omega_\fK} g\,\fR f\,\d\Q = \int_{\Omega_\fK} f\,\fR g\,\d\Q.
	\end{equation}
	Using \eqref{eqn:transition-kernel-environment-chain}, we get
	\begin{equation}
		\label{eqn:L2-inner-product-expression}
		\int_{\Omega_\fK} g\,\fR f\,\d\Q = \tfrac{1}{1+\bar{\E}[M_0/N_0]}[I_1(f,g)+I_2(f,g)],
	\end{equation} where
	\begin{equation}
		\label{eqn:reversible-I_1-and-I_2}
		\begin{aligned}
			I_1(f,g) &= q_s\int_{\cE_\fK}g(\fe,1)f(\fe,0)\,\d\bP(\fe)+(1-q_s)\int_{\cE_\fK}\Big[\sum_{j\in\Z^d}\hat{p}(j)g(\fe,1)f(T_j\fe,1)\Big]\,\d\bP(\fe),\\
			I_2(f,g) &= q_s\int_{\cE_\fK}g(\fe,0)[f(\fe,1)-f(\fe,0)]\,\d\bP(\fe)+\int_{\cE_\fK}\tfrac{q_s}{\omega(0)}\,g(\fe,0)f(\fe,0)\,\d\bP(\fe).
		\end{aligned}
	\end{equation}
	Note that, by condition (2) in Assumption~\ref{assumpt:clustering-env}, we have $\hat{p}(k)=\hat{p}(-k)$ for all $k\in\Z^d$, and so by translation invariance of $\bP$ the second term in $I_1(f,g)$ remains unchanged if we interchange $f$ and $g$. Indeed,
	\begin{equation}
		\begin{aligned}
			\int_{\cE_\fK}\Big[\sum_{j\in\Z^d}&\hat{p}(j)g(\fe,1)f(T_j\fe,1)\Big]\,\d\bP(\fe) = \int_{\cE_\fK}\Big[\sum_{j\in\Z^d}\hat{p}(j)g(T_{-j}\fe,1)f(\fe,1)\Big]\,\d\bP(\fe)\\
			&=\int_{\cE_\fK}\Big[\sum_{j\in\Z^d}\hat{p}(-j)g(T_j\fe,1)f(\fe,1)\Big]\,\d\bP(\fe)\\
			&= \int_{\cE_\fK}\Big[\sum_{j\in\Z^d}\hat{p}(j)g(T_j\fe,1)f(\fe,1)\Big]\,\d\bP(\fe),\quad(\text{by symmetry of }\hat{p}(\cdot)).
		\end{aligned}
	\end{equation}
	Thus, using \eqref{eqn:reversible-I_1-and-I_2} and the above, we see that $I_1(f,g)+I_2(f,g)=I_1(g,f)+I_2(g,f)$, which combined with \eqref{eqn:L2-inner-product-expression} proves the claim in \eqref{eqn:reversibility-criterion}.
\end{proof}
\subsection[An application: strong law of large numbers]{An application: strong law of large numbers}
\label{sec-law-of-large-number}
As pointed out earlier in Remark~\ref{rem:validity-in-all-dimension}, part (1) of Theorem~\ref{thm:invariant-distribution-of-W} holds in any dimension $d\geq 1$, even when the migration kernel $\hat{p}(\cdot)$ (see \eqref{eqn:transformation-of-parameters}) is not symmetric. An interesting application of this theorem is the strong law of large numbers stated below.
\begin{theorem}{\bf[Strong law of large numbers]}
	\label{thm:law-of-large-numbers}
	Let $\widehat{\Theta}^\fe=(X_n^\fe,\alpha_n^\fe)_{n\in\N_0}$ be the subordinate Markov chain evolving in environment $\fe$ with law $\widehat{P}_{(0,\alpha)}^\fe$ (see Definition~{\rm \ref{defn:discrete-coordinate-process}}),  and let $\bP$ be the translation-invariant, ergodic field as in Assumption~{\rm\ref{assump:environment-law}}. Assume that the migration kernel $\hat{p}(\cdot)$ (see \eqref{eqn:transformation-of-parameters}) has finite range and mean 
	\begin{equation}
		v:=\sum_{j\in\Z^d} j\,\hat{p}(j).
	\end{equation}
	Then, for $\bP$-almost every realization of $\fe$ and $\alpha\in\{0,1\}$,
	\begin{equation}
		\lim\limits_{n\to\infty}\frac{X_n^\fe}{n}=\frac{1-q_s}{1+\rho}\, v\quad\widehat{P}_{(0,\alpha)}^\fe \text{ a.s.},
	\end{equation}
	where 
	$\rho:=\bar{\E}\big[\tfrac{M_0}{N_0}\big]$ and $q_s$ is as in \eqref{eqn:transformation-of-parameters}.
\end{theorem}
Recall that $X_n^\fe$ denotes the location in $\Z^d$ at time $n$ of a particle that evolves according to the subordinate Markov chain $\widehat{\Theta}^\fe$ in environment $\fe$. Therefore, the intuitive meaning of the above result is that the particle on average spends a $\tfrac{1}{1+\rho}$ fraction of its time in the active state, and since it migrates only while being active with probability $1-q_s$, the overall velocity is scaled by the factor $\tfrac{1-q_s}{1+\rho}$.
\begin{remark}{\bf [Transference of law of large numbers]}
	Using Theorem~\ref{thm:law-of-large-numbers}, Lemma~\ref{lemma:poisson-clock-representation} and the elementary renewal theorem, we can \emph{transfer} the law of large numbers on $\widehat{\Theta}^\fe$ to the continuous-time process $\Theta^\fe=(x_t^\fe,\alpha_t^\fe)_{t\geq 0}$ (see Definition~\ref{defn:single-particle-dual}) and obtain, for $\bP$-almost every realization of $\fe$ and $\alpha\in\{0,1\}$,
	\begin{equation}
		\lim\limits_{t\to\infty}\tfrac{x_t^\fe}{t}=\tfrac{1}{1+\rho} \sum_{j\in\Z^d} j\,a(0,j),\quad P_{(0,\alpha)}^\fe \text{-a.s.}
	\end{equation}
\end{remark} 
We conclude this section with the proof of the above theorem. The proof is based on an application of the classical Birkhoff pointwise ergodic theorem combined with the Azuma inequality for martingales having bounded increments.
\begin{proof}[Proof of Theorem~{\rm\ref{thm:law-of-large-numbers}}]
	Following the standard route as taken in \cite[Lecture 1]{Sznitman02}, we start by defining a ($d$-dimensional) Martingale $M^\fe:=(M_n^\fe)_{n\in\N}$ constructed from the ``local drift'' of a particle moving in an environment $\fe\in\cE_\fK$ according to the subordinate Markov chain $\widehat{\Theta}^\fe=(X_n^\fe,\alpha_n^\fe)_{n\in\N_0}$ with law $\widehat{P}^\fe_{(0,\alpha)}$. With this aim, let us fix $\fe\in\cE_\fK$, $\alpha\in\{0,1\}$ and set $M_0^\fe:=X_0^\fe$. For $n\in\N$, define
	\begin{equation}
		\label{eqn:defn-martingale-LLN}
		M_n^\fe: = X_n^\fe - (1-q_s)v\sum_{l=0}^{n-1}\alpha_n^\fe.
	\end{equation}
We show that $M^\fe$ is a martingale (viewed component-wise) under the law $\widehat{P}^\fe_{(0,\alpha)}$ w.r.t.\ the natural filtration $(\mathcal{F}_n)_{n\in\N_0}$ of the subordinate Markov chain $\widehat{\Theta}^\fe$. {\corrected Indeed, if $\widehat{\mathcal{J}}_\fe$ denotes the discrete Markov generator of $\widehat{\Theta}^\fe$ and $h\,:G\to\Z^d$ is the projection onto the first coordinate, (i.e., the map $(i,\alpha)\mapsto i$), then the action of $\widehat{\mathcal{J}}_\fe$ on $h$ is given by
	\begin{equation}
		\label{eqn:generator-action-on-projection}
		\begin{aligned}
			(\widehat{\mathcal{J}}_\fe h)(i,\alpha)&=
			\widehat{E}_{(i,\alpha)}^\fe\Big[h(X_1^\fe,\alpha_1^\fe) - h(X_0^\fe,\alpha_0^\fe)\Big]=\widehat{E}_{(i,\alpha)}^\fe\big[X_1^\fe - X_0^\fe\big]\\
			&= \alpha\Big[q_s i+(1-q_s)\sum_{j\in\Z^d}\hat{p}(j)(i+j)-i\Big]\\
			&=\alpha(1-q_s)v,
		\end{aligned}
	\end{equation}
	where we used \eqref{eqn:trans-prob-of-aux-proc}--\eqref{eqn:transformation-of-parameters} for the computation. Therefore, combining \eqref{eqn:defn-martingale-LLN}--\eqref{eqn:generator-action-on-projection}, we get that
	\begin{equation}
		M_n^\fe = h(X_n^\fe,\alpha_n^\fe) - \sum_{l=0}^{n-1}(\widehat{\mathcal{J}}_\fe h)(X_l^\fe,\alpha_l^\fe),\quad n\in\N,
	\end{equation}
	which is nothing else but the so-called Dynkin's martingale for the Markov chain $\widehat{\Theta}^\fe$ (see e.g., \cite[Proposition~6.1.1]{Lawler2009}).} Also note that it has bounded increments by virtue of the finite-range assumption on the migration kernel $\hat{p}(\cdot)$. A standard application of Azuma inequality and the Borel Cantelli lemma yield (see e.g., \cite[Lecture 1, page 14]{Sznitman02})
\begin{equation}
	\lim\limits_{n\to\infty}\frac{M_n^\fe}{n} = 0\quad\widehat{P}^\fe_{(0,\alpha)}\text{-a.s.}
\end{equation}
Observe from above and \eqref{eqn:defn-martingale-LLN}, the proof will be complete if we prove the following: for $\bP$-almost every $\fe\in\cE_\fK$ and any $\alpha\in\{0,1\}$,
\begin{equation}
	\label{eqn:constant-activity-time}
\lim\limits_{n\to\infty}\frac{1}{n}\sum_{l=0}^{n-1}\alpha_l^\fe = \frac{1}{1+\bar{\E}[M_0/N_0]},\quad\widehat{P}^\fe_{(0,\alpha)}\text{ a.s.}
\end{equation}
This is a consequence of Birkhoff's pointwise ergodic theorem. Indeed, let $\tilde{P}_\Q$ be the canonical law defined on the path space $\Omega_\fK^{\N_0}$ of the auxiliary environment process $W=(\fe_n,\alpha_n)_{n\in\N_0}$ (recall Definition~\ref{defn:environment-chain}) with initial distribution $\Q$ (see \eqref{eqn:defn-of-invariant-dist}). In other words,
\begin{equation}
	\label{eqn:defn-of-canonical-law}
	\tilde{P}_\Q(W\in \,\cdot\,) := \int_{\Omega_\fK}\widehat{P}^\fe_{(0,\alpha)}(W\in\,\cdot\,)\,\d\Q(\fe,\alpha).
\end{equation}
Let $S\,:\Omega_\fK^{\N_0}\to \Omega_\fK^{\N_0}$ be the natural left-shift operator and $f\,:\Omega_\fK^{\N_0}\to\{0,1\}$ be the function
\begin{equation}
	(\mathfrak{a}_n,\beta_n)_{n\in\N_0}\mapsto \beta_0,\quad (\mathfrak{a}_n,\beta_n)_{n\in\N_0}\in\Omega_\fK^{\N_0}.
\end{equation}
Since, by part (1) of Theorem~\ref{thm:invariant-distribution-of-W}, $\Q$ is a stationary and ergodic distribution of $W$, we see that $S$ is a measure-preserving ergodic transformation of the dynamical system $(\Omega_\fK^{\N_0},\tilde{P}_\Q)$. Applying Birkhoff's pointwise ergodic theorem to the bounded function $f$, we obtain
\begin{equation}
	\label{eqn:application-of-Birkhoff}
	\lim\limits_{n\to\infty}\frac{1}{n}\sum_{l=0}^{n-1}f\circ S^l = \int_{\Omega_\fK^{\N_0}} f\,\d\tilde{P}_\Q = \int_{\Omega_\fK}\widehat{E}^\fe_{(0,\alpha)}\big[f((\fe_n,\alpha_n)_{n\in\N_0})\big]\,\d\Q(\fe,\alpha),
\end{equation}
where the first equality holds $\tilde{P}_\Q$-almost everywhere and the second equality follows from \eqref{eqn:defn-of-canonical-law}. We compute the left and the right side of \eqref{eqn:application-of-Birkhoff} using the definition of $f$ and  \eqref{eqn:defn-of-invariant-dist}, to obtain
\begin{equation}
	\begin{aligned}
		\lim\limits_{n\to\infty}\frac{1}{n}\sum_{l=0}^{n-1}\beta_l = \int_{\Omega_\fK}\alpha\,\d\Q(\fe,\alpha)=\frac{1}{1+\bar{\E}[M_0/N_0]}
	\end{aligned}
\end{equation}
for $\tilde{P}_\Q$-almost every $(\mathfrak{b}_l,\beta_l)_{l\in\N_0}$. However, \eqref{eqn:defn-of-canonical-law} combined with the above implies that \eqref{eqn:constant-activity-time} holds for all $(\fe,\alpha)\in A$ for some $A\in\Sigma$ such that $\Q(A)=1$. The result now follows from the equivalence of $\Q$ and $\bP$ stated in Lemma~\ref{lemma:equivalence-of-pi-P}.
\end{proof}
\section{Fundamental theorem of Markov chains}
\label{sec-fundamental-theorem-of-MC}
In this section we provide the proof of Proposition~\ref{prop:fundamental-theorem-of-MC}. Let us recall the statement of the proposition for convenience of the reader.
\begin{proposition}
	Let $(\Omega,\Sigma,\Q)$ be a probability space, where the $\sigma$-field $\Sigma$ is countably generated. Let $W := (W_n)_{n\in\N_0}$ be a Markov chain on the state space $\Omega$, and assume that $\Q$ is a reversible and ergodic stationary distribution for $W$. If $-1$ is not an eigenvalue of the Markov kernel operator $\fR\,:L_\infty(\Omega,\Q)\to L_\infty(\Omega,\Q)$ associated to $W$, then for every bounded measurable function $f \in \mathcal{F}_b(\Omega)$ and $\Q$-almost every $w\in\Omega,$
	\begin{equation}
		\label{eqn:appen-convergence-of-expectation}
		\lim\limits_{n\to\infty}\E_w[f(W_n)] = \int_{\Omega}f\,\d\Q,
	\end{equation}
	where the expectation on the left is taken w.r.t.\ the law of $W$ started at $w$.
\end{proposition}
\begin{proof}
	If $Q(\cdot\,,\cdot)$ denotes the transition kernel of $W$, the action of the Markov operator $\fR$ on $f\in L_\infty(\Omega,\Q)$ is well-defined and is given by
	\begin{equation}
	\fR f(w) := \int_\Omega f(y)\,Q(w,\d y).
	\end{equation}
In fact, since $\Q$ is invariant for $W$, the same definition extends $\fR$ in a canonical way to a positive contraction operator on $L_p(\Omega,\Q)$ for any $p\geq 1$. Furthermore, by reversibility of $\Q$, the operator $\fR$ becomes self-adjoint on $L_2(\Omega,\Q) $ as well. Let $f\in\mathcal{F}_b(\Omega)$ be fixed. Because $f$ is bounded, $\fR^nf\in \mathcal{F}_b(\Omega)$, and by the Markov property of $W$ it follows that
\begin{equation}
	\label{eqn:expectaion-in-markov-operator}
	\E_w[f(W_n)] = \fR^nf(w),\quad w\in\Omega,\, n\in \N.
\end{equation}
Because $\fR$ is self-adjoint, we see that $\fR^2$ is a nonnegative-definite operator on the Hilbert space $L_2(\Omega,\Q)$ equipped with the natural $L_2$ inner product, and this allows us to conclude from \cite[Corollary 3]{Stein61} (see also \cite[Theorem 1]{Rota62}) that there exist $\psi,\widehat{\psi}\in L_2(\Omega,\Q)$ satisfying
\begin{equation}
	\label{eqn:convergence-of-odd-even-subseq}
		\psi = \lim\limits_{n\to\infty}\fR^{2n}f,\quad\widehat{\psi} =\lim\limits_{n\to\infty}\fR^{2n+1}f,
\end{equation}
where the convergence is in $L_2$-norm and $\Q$-almost everywhere. It is worth mentioning that the convergence in \eqref{eqn:convergence-of-odd-even-subseq}, which follows from \cite[Corollary 3]{Stein61}, essentially uses the classical Banach principle (see e.g., \cite{Bello66}) along with a maximal ergodic inequality. The convergence, in fact, holds for any function in $(L\log^+L)(\Omega,\Q)$. By the almost sure convergence of $ \fR^{2n}f$ (resp.\ $\fR^{2n+1}f$) and the $L_\infty$ contractivity of $\fR$, we see that $\psi,\widehat{\psi}\in L_\infty(\Omega,\Q)$ as well. The $L_2$ contractivity of the linear operator $\fR$ also implies that, 
\begin{equation}
\psi = \fR\widehat{\psi},\quad \widehat{\psi} = \fR\psi\quad \Q\text{-a.s.,}
\end{equation}
from which we get
\begin{equation}
	\label{eqn:ergodic-square-identity}
	\fR^2\psi = \psi,\ \fR^2\widehat{\psi}=\widehat{\psi},\quad \Q \text{ a.s.}
\end{equation}
We claim that if $-1$ is not an eigenvalue of $\fR$ as an operator on $L_\infty(\Omega,\Q)$, then we must have
\begin{equation}
	\label{eqn:psi-equal-psi-hat}
	\psi = \widehat{\psi} = \int_\Omega f\,\d\Q,\quad \Q \text{ a.s.}
\end{equation}
Note that \eqref{eqn:appen-convergence-of-expectation} will follow once we prove \eqref{eqn:psi-equal-psi-hat}. Indeed, \eqref{eqn:psi-equal-psi-hat} combined with \eqref{eqn:expectaion-in-markov-operator}--\eqref{eqn:convergence-of-odd-even-subseq} implies that, for $\Q$-almost every $w\in\Omega$, both the odd and even subsequence of $(\E_w[f(W_n)])_{n\in\N}$ converge to the same limit $\Q(f):=\int_{\Omega}f\,\d\Q,$ which necessarily forces the convergence of $\E_w[f(W_n)]$ to $\Q(f)$ as $n\to\infty$.

To prove \eqref{eqn:psi-equal-psi-hat}, it suffices to show that $\psi $ and $\widehat{\psi}$ are constant $\Q$ a.s., because by the invariance of $\fR$ w.r.t.\ $\Q$ and bounded convergence we have \begin{equation}
	\label{eqn:constant-integral}
	\int_\Omega \psi \,\d\Q = \lim\limits_{n\to\infty}\int_\Omega\fR^{2n}f\,\d\Q = \int_\Omega f\,\d\Q = \lim\limits_{n\to\infty}\int_{\Omega}\fR^{2n+1}f\,\d\Q = \int_\Omega\widehat{\psi}\,\d\Q.
\end{equation}
We only prove the claim for $\psi$, as the same argument works for $\widehat{\psi}$. Let us set $g := \fR\psi - \psi$. From \eqref{eqn:ergodic-square-identity}, we have
\begin{equation}
	\label{eqn:square-eigenfunction}
	\fR g = - g,\quad \Q \text{ a.s.},
\end{equation}
and also $\|g\|_\infty \leq 2\|\psi\|_\infty < \infty$. Thus, $g\in L_\infty(\Omega,\Q)$ is such that $\Q$-a.s.\ $\fR g = -g$, and hence by our assumption we must have $g=0$ a.s. In other words, $\Q$-a.s.\ $\fR\psi-\psi = 0$ and therefore ergodicity of $\fR$ in $L_2(\Omega,\Q)$, which is equivalent to the ergodicity of $W$ under $\Q$, implies that $\psi$ is necessarily a constant $\Q$ a.s.
\end{proof}

\begin{thebibliography}{56}
	
	\bibitem{Bello66}
	A.~Bellow and R.~L. Jones, A {B}anach principle for ${L}_\infty$,
	\emph{Adv. Math.} \textbf{120} (1996), 155–172. \MR{1392277}
	
	\bibitem{Birkner21}
	S.~A. Bethuelsen, M.~Birkner, A.~Depperschmidt, and T.~Schlüter, Local
	limit theorems for a directed random walk on the backbone of a supercritical
	oriented percolation cluster, 2021, {P}reprint. \href{https://arxiv.org/abs/2105.09030}{arXiv:2105.09030}
	
	\bibitem{Biskup11}
	M.~Biskup, Recent progress on the random conductance model,
	\emph{Probab. Surveys} \textbf{8} (2011), 294--373. \MR{861133}
	
	\bibitem{Blath19}
	J.~Blath, E.~Buzzoni, A.~G. Casanova, and M.~Wilke-Berenguer, Structural
	properties of the seed bank and the two island diffusion, \emph{J. Math.
		Biol.} \textbf{79} (2019), no.~1, 369--392. \MR{3975877}
	
	\bibitem{Blath13}
	J.~Blath, A.~G. Casanova, N.~Kurt, and D.~Spano, The ancestral process of
	long-range seed bank models, \emph{J. Appl. Probab.} \textbf{50} (2013),
	no.~3, 741--759. \MR{3102512}
	
	\bibitem{BCKW16}
	J.~Blath, A.~G. Casanova, N.~Kurt, and M.~Wilke-Berenguer, A new
	coalescent for seed-bank models, \emph{Ann. Appl. Probab.} \textbf{26}
	(2016), no.~2, 857--891. \MR{3476627}
	
	\bibitem{BCEK15}
	J.~Blath, B.~Eldon, A.~G. Casanova, and N.~Kurt, Genealogy of a
	{W}right-{F}isher {M}odel with {S}trong {S}eed {B}ank {C}omponent,
	\emph{{XI} Symposium on Probability and Stochastic Processes}, Springer
	International Publishing, 2015, pp.~81--100. \href{https://doi.org/10.1007/978-3-319-13984-5_4}{\nolinkurl{doi: 10.1007/978-3-319-13984-5_4}}
	
	\bibitem{Blath21}
	J.~Blath, F.~Hermann, and M.~Slowik, A branching process model for
	dormancy and seed banks in randomly fluctuating environments, \emph{J. Math.
		Biol.} \textbf{83} (2021), no.~2, 1--40. \MR{4289473}
	
	\bibitem{BK}
	J.~Blath and N.~Kurt, Population genetic models of dormancy,
	\emph{Probabilistic Structures in Evolution}, European Mathematical Society
	Publishing House, 2021, pp.~247--265. \MR{4331863}
	
	\bibitem{Bolt20}
	E.~Bolthausen and I.~Goldsheid, Recurrence and transience of random walks
	in random environments on a strip, \emph{Commun. Math. Phys.} \textbf{214}
	(2000), no.~2, 429--447. \MR{1796029}
	
	\bibitem{Sznitman02}
	E.~Bolthausen and A.~Sznitman, Ten lectures on random media, vol.~32,
	Springer Science \& Business Media, 2002. \MR{1890289}
	
	\bibitem{Burkholder61}
	B.~L. Burkholder and Y.~S. Chow, Iterates of conditional expectation
	operators, \emph{Proc. Am. Math. Soc.} \textbf{12} (1961), no.~3, 490--495. \MR{0142144}
	
	\bibitem{CGGR}
	G.~Carinci, C.~Giardin{\`a}, C.~Giberti, and F.~Redig, Dualities in
	population genetics: A fresh look with new dualities, \emph{Stoch. Proc.
		Appl.} \textbf{125} (2015), 941--969. \MR{3303963}
	
	\bibitem{Casanova18}
	A.~G. Casanova and D.~Span{\`o}, Duality and fixation in
	${\Xi}$-{W}right--{F}isher processes with frequency-dependent selection,
	\emph{Ann. Appl. Probab.} \textbf{28} (2018), no.~1, 250--284. \MR{3770877}
	
	\bibitem{Maite19}
	A.~G. Casanova, D.~Span{\`o}, and M.~Wilke-Berenguer, The effective
	strength of selection in random environment, 2019, Preprint. \href{https://arxiv.org/abs/1903.12121}{arXiv:1903.12121}
	
	\bibitem{Cohen66}
	D.~Cohen, Optimizing reproduction in a randomly varying environment,
	\emph{J. Theor. Biol.} \textbf{12} (1966), no.~1, 119--129. \href{https://doi.org/10.1016/0022-5193(66)90188-3}{\nolinkurl{doi: 10.1016/0022-5193(66)90188-3}}
	
	\bibitem{Cohen14}
	G.~Cohen, C.~Cuny, and M.~Lin, Almost everywhere convergence of powers of
	some positive ${L}_p$ contractions, \emph{J. Math. Anal.} \textbf{420}
	(2014), no.~2, 1129--1153. \MR{3240070}
	
	\bibitem{Hollander88}
	F.~den Hollander, Mixing properties for random walk in random scenery,
	\emph{Ann. Probab.} \textbf{16} (1988), no.~4, 1788 -- 1802. \MR{0958216}
	
	\bibitem{HP}
	F.~den Hollander and Pederzani. G., Multi-colony wright-fisher with
	seed-bank, \emph{Indag. Math.} \textbf{28} (2017), no.~3, 637--669. \MR{3651340}
	
	\bibitem{HN01}
	F.~den Hollander and S.~Nandan, Spatially inhomogeneous populations with
	seed-banks: {I}. duality, existence and clustering, \emph{J. Theor. Probab.} \textbf{35} (2022), no.~3, 1795--1841. \MR{4488559}
	
	\bibitem{HN02}
	F.~den Hollander and S.~Nandan, Spatially inhomogeneous populations with seed-banks: {II}.
	clustering regime, \emph{Stoc. Proc. Appl.} \textbf{150} (2022), 116--146. \MR{4419571}
	
	\bibitem{Nandan2023}
	S.~Nandan, Dormancy in {S}tochastic {I}nteracting {S}ystems, {D}octoral {T}hesis, Leiden University, 2023, pp. 1--238.
	
	\bibitem{Steif2006}
	F.~den Hollander and J.~E. Steif, Random walk in random scenery: A survey
	of some recent results, Institute of Mathematical Statistics Lecture Notes -
	Monograph Series, Institute of Mathematical Statistics, 2006, pp.~53--65. \MR{2306188}
	
	\bibitem{Goldshied19}
	D.~Dolgopyat and I.~Goldsheid, Invariant measure for random walks on
	ergodic environments on a strip, \emph{Ann. Probab.} \textbf{47} (2019),
	no.~4, 2494--2528. \MR{3980926}
	
	\bibitem{Ilya19}
	D.~Dolgopyat and I.~Goldsheid, Local limit theorems for random walks in a random environment on
	a strip, 2019, {P}reprint. \href{https://arxiv.org/abs/1910.12961}{arXiv:1910.12961}
	
	\bibitem{Uphill22}
	S.~Floreani, C.~Giardin{\`{a}}, F.~den Hollander, S.~Nandan, and F.~Redig,
	Switching interacting particle systems: Scaling limits, uphill
	diffusion and boundary layer, \emph{J. Stat. Phys.} \textbf{186} (2022),
	no.~3, 1--33. \MR{4372572}
	
	\bibitem{GHO1}
	A.~Greven, F.~den Hollander, and M.~Oomen, Spatial populations with
	seed-bank: well-posedness, duality and equilibrium, \emph{Electron. J.
		Probab.} \textbf{27} (2022), 1--88. \MR{4377129}
	
	\bibitem{GHO3}
	A.~Greven, F.~den Hollander, and M.~Oomen, Spatial populations with seed-bank: renormalization on the
	hierarchical group, 2022, {P}reprint. \href{https://arxiv.org/abs/2110.02714}{arXiv:2110.02714}
	
	\bibitem{Iwanik86}
	A.~Iwanik and R.~Shiflett, The root problem for stochastic and doubly
	stochastic operators, \emph{J. Math. Anal.} \textbf{113} (1986), no.~1,
	93--112. \MR{0826661}
	
	\bibitem{JK2014}
	S.~Jansen and N.~Kurt, On the notion(s) of duality for {M}arkov
	processes, \emph{Probab. Surveys} \textbf{11} (2014), 59--120. \MR{3201861}
	
	\bibitem{Kaj2001}
	I.~Kaj, S.~M. Krone, and M.~Lascoux, Coalescent theory for seed bank
	models, \emph{J. Appl. Probab.} \textbf{38} (2001), no.~2, 285--300. \MR{1834743}
	
	\bibitem{kallenberg97}
	O.~Kallenberg, Foundations of {M}odern {P}robability, vol.~2, Springer,
	1997. \MR{1464694}
	
	\bibitem{Kesten77}
	H.~Kesten, A renewal theorem for random walk in a random environment,
	\emph{Proc. Sympos.} \textbf{31} (1977), 67--77. \MR{0458648}
	
	\bibitem{Kon11}
	I.~Kontoyiannis and S.~P. Meyn, Geometric ergodicity and the spectral gap
	of non-reversible {M}arkov chains, \emph{Probab. Theory. Relat. Fields.}
	\textbf{154} (2011), no.~1-2, 327--339. \MR{2981426}
	
	\bibitem{Kozlov85}
	S.~M. Kozlov, The averaging method and walks in inhomogeneous
	environments, \emph{Uspekhi Mat. Nauk} \textbf{40} (1985), no.~2(242), 61--120,
	238. \MR{786087}
	
	\bibitem{Kulik7}
	A.~Kulik and M.~Scheutzow, Generalized couplings and convergence of
	transition probabilities, \emph{Probab. Theory. Relat. Fields.} \textbf{171}
	(2017), no.~1-2, 333--376. \href{https://doi.org/10.1007/s00440-017-0779-8}{\nolinkurl{doi: 10.1007/s00440-017-0779-8}}
	
	\bibitem{Lally86}
	S.~Lalley, An extension of {K}esten's renewal theorem for random walk in
	a random environment, \emph{Adv. Appl. Math.} \textbf{7} (1986), no.~1,
	80--100. \MR{0834222}
	
	\bibitem{Lawler2009}
	G.~F. Lawler and V.~Limic, Random walk: A {M}odern {I}ntroduction, Cambridge
	University Press, 2009. \MR{2677157}
	
	\bibitem{LHBB}
	J.~T. Lennon, F.~den Hollander, M.~Wilke-Berenguer, and J.~Blath,
	Principles of seed banks and the emergence of complexity from
	dormancy, \emph{Nat. Comm.} \textbf{12} (2021), no.~1, 4807. \href{https://doi.org/10.1038/s41467-021-24733-1}{\nolinkurl{doi: 10.1038/s41467-021-24733-1}}
	
	\bibitem{LJ11}
	J.~T. Lennon and S.~E. Jones, Microbial seed banks: the ecological and
	evolutionary implications of dormancy, \emph{Nat. Rev. Microbiol.}
	\textbf{9} (2011), 119--130. \href{https://doi.org/10.1038/nrmicro2504}{\nolinkurl{doi: 10.1038/nrmicro2504}}
	
	\bibitem{L}
	T.~M. Liggett, Interacting Particle Systems, Springer, Berlin,
	1985.
	
	\bibitem{Liggett10}
	T.~M. Liggett, {C}ontinuous {T}ime {M}arkov {P}rocesses: {A}n {I}ntroduction,
	vol. 113, American Mathematical Soc., 2010.
	
	\bibitem{Lin82}
	M.~Lin, On the ``zero-two'' law for conservative {M}arkov processes,
	\emph{Probab. Theory Relat. Fields} \textbf{61} (1982), no.~4, 513--525. \MR{0682577}
	
	\bibitem{Meyn93}
	S.~P. Meyn and R.~L. Tweedie, Markov {C}hains and {S}tochastic {S}tability,
	Springer, London, 1993.
	
	\bibitem{Nummelin84}
	E.~Nummelin, General {I}rreducible {M}arkov {C}hains and {N}on-{N}egative
	{O}perators, Cambridge University Press, 1984.
	
	\bibitem{Oomen22}
	M.~Oomen, Spatial {P}opulations with {S}eed-{B}ank, {D}octoral {T}hesis,
	Leiden University, 2022, pp.~1--368.
	
	\bibitem{Ornstein68}
	D.~Ornstein, On the pointwise behavior of iterates of a self-adjoint
	operator, \emph{J. Math. Mech.} \textbf{18} (1968), no.~5, 473--477. \MR{0236354}
	
	\bibitem{Rob2004}
	G.~O. Roberts and J.~S. Rosenthal, General state space {M}arkov chains
	and {MCMC} algorithms, \emph{Probab. Surv.} \textbf{1} (2004), 20--71. \MR{2095565}
	
	\bibitem{Rota62}
	G.~Rota, An “alternierende verfahren” for general positive
	operators, \emph{ Bull. Am. Math. Soc.} \textbf{68} (1962), no.~2, 95--102. \MR{0133847}
	
	\bibitem{Shiga80}
	T.~Shiga, An interacting system in population genetics, \emph{Kyoto J.
		Math.} \textbf{20} (1980), no.~2, 213--242. \MR{0582165}
	
	\bibitem{Shiga80-1}
	T.~Shiga, An interacting system in population genetics, {II}, \emph{Kyoto
		J. Math.} \textbf{20} (1980), no.~4, 723--733. \MR{0592356}
	
	\bibitem{LS}
	W.~Shoemaker and J.~T. Lennon, Evolution with a seed bank: The population
	genetic consequences of microbial dormancy, \emph{Evol. Appl.} \textbf{11}
	(2018), no.~1, 60--75. \href{https://doi.org/10.1111/eva.12557}{\nolinkurl{doi: 10.1111/eva.12557}}
	
	\bibitem{Stein61}
	E.~M. Stein, On the maximal ergodic theorem, \emph{Proc. Natl. Acad.
		Sci. U.S.A.} \textbf{47} (1961), no.~12, 1894--1897. \MR{0131517}
	
	\bibitem{Tellier11}
	A.~Tellier, S.~Laurent, H.~Lainer, P.~Pavlidis, and W.~Stephan, Inference
	of seed bank parameters in two wild tomato species using ecological and
	genetic data, \emph{Proc. Natl. Acad. Sci.} \textbf{108} (2011), no.~41,
	17052--17057. \href{https://doi.org/10.1073/pnas.1111266108}{\nolinkurl{doi: 10.1073/pnas.1111266108}}
	
	\bibitem{Wisnoski22}
	N.~I. Wisnoski and L.~G. Shoemaker, Seed banks alter metacommunity
	diversity: The interactive effects of competition, dispersal and dormancy,
	\emph{Ecol. Lett} \textbf{25} (2022), no.~4, 740--753. \href{https://doi.org/10.1111/ele.13944}{\nolinkurl{doi: 10.1111/ele.13944}}
	
	\bibitem{Ivkovi12}
	D.~{\v{Z}}ivkovi{\'{c}} and A.~Tellier, Germ banks affect the inference
	of past demographic events, \emph{Mol. Ecol.} \textbf{21} (2012), no.~22,
	5434--5446. \href{https://doi.org/10.1111/mec.12039}{\nolinkurl{doi: 10.1111/mec.12039}}
	
\end{thebibliography}

\providecommand{\MR}[1]{%
	\href{https://mathscinet.ams.org/mathscinet-getitem?mr=#1}{MR#1}}

\end{document}